\documentclass[10pt]{amsart}

\usepackage{color}

\usepackage{amsmath,amsthm,amssymb,bbm}

\usepackage{mathptmx}

\usepackage{url,layout}

\theoremstyle{definition}

\newtheorem*{thm*}{Theorem}

\newtheorem{prop}{Proposition}[section]

\newtheorem{thm}[prop]{Theorem}

\newtheorem{lem}[prop]{Lemma}

\newtheorem{df}[prop]{Definition}

\newtheorem{rem}[prop]{Remark}

\newtheorem{cor}[prop]{Corollary}

\newcommand{\supp}{\mathop{ {\mathrm{supp} }}\nolimits}

\newcommand{\Tor}{\mathop{\mathrm{Tor}}\nolimits}

\newcommand{\im}{\mathop{\mathrm{im}}\nolimits}

\newcommand{\id}{\mathop{\mathrm{id}}\nolimits}

\begin{document}
\title[$L^2$-Betti numbers and costs of groupoids]{$L^2$-Betti numbers and costs in the framework \\ of discrete groupoids}
\author{Atsushi Takimoto}
\address{Graduate School of Mathematics, Kyushu University\\ Fukuoka, 819-0395, Japan}
\email{a-takimoto@math.kyushu-u.ac.jp}

\begin{abstract}
We unify the known basic theories on $L^2$-Betti numbers and costs in the framework of probability measure preserving discrete groupoids.  
\end{abstract}

\maketitle

\section{Introduction} 

There are two approaches to the {\it $L^2$-Betti numbers} $\beta_n^{(2)}(\Gamma)$, $n=0,1,2,\dots$, of an arbitrary (countable) discrete group $\Gamma$; one is geometric and the other is algebraic, each of which has individual merits. The first and geometric one due to Cheeger and Gromov \cite{cheeger-gromov} utilizes chain complexes of Hilbert spaces obtained from appropriate simplicial complexes equipped with actions of $\Gamma$, while the second and algebraic one due to L{\"{u}}ck (see his book \cite{luck:survey}) does chain complexes of algebraic $\Gamma$-modules with the help of his `algebraization' of the original Murray-von Neumann dimension.  

Following Cheeger-Gromov's geometric approach, Gaboriau \cite{gaboriau:betti} introduced the $L^2$-Betti numbers $\beta_n^{(2)}(\mathcal{R})$ of an arbitrary probability measure preserving (pmp for short) (countable) discrete equivalence relation $\mathcal{R}$. For an arbitrary essentially free, pmp action $\Gamma\curvearrowright (X,\mu)$ of a discrete group he showed, among others, that its orbit equivalence relation $\mathcal{R}_{\Gamma\curvearrowright (X,\mu)}$ satisfies the formula 
\begin{equation}\label{a}
\beta_n^{(2)}(\mathcal{R}_{\Gamma\curvearrowright (X,\mu)}) = \beta_n^{(2)}(\Gamma),
\end{equation}
which in turn says that the $\beta_n^{(2)}(\Gamma)$ are orbit equivalence invariants. Under the influence of Gaboriau's work, Sauer \cite{sauer} then adapted L\"{u}ck's algebraic approach to an arbitrary pmp discrete groupoid $G$, and defined the $L^2$-Betti numbers $\beta_n^{(2)}(G)$. The pmp discrete groupoids form a natural class including both the discrete groups and the pmp discrete equivalence relations as its subclasses. By definition, Sauer's $\beta_n^{(2)}(G)$ recovers $\beta_n^{(2)}(\Gamma)$ when $G$ is a discrete group $\Gamma$. Moreover, it is rather easier to prove the formula \eqref{a} in his definition, and it turns out that Sauer's $L^2$-Betti numbers agree with Gaboriau's when $G = \mathcal{R}_{\Gamma\curvearrowright (X,\mu)}$ with essentially free, pmp actions $\Gamma\curvearrowright (X,\mu)$. The complete identification between Gaboriau's and Sauer's $L^2$-Betti numbers for pmp discrete equivalence relations was finally settled by Neshveyev and Rustad \cite{neshveyev}. Their proof utilizes more recent technologies developed by Thom \cite{thom}, and turns out to simplify some technical parts of Gaboriau's theory. However, it is still missing to develop the geometric approach to the $L^2$-Betti numbers in the framework of pmp discrete groupoids, and we will fill up this gap in the present notes.     

Before his introduction of $L^2$-Betti numbers of pmp discrete equivalence relations, Gaboriau \cite{gaboriau:cost} studied the so-called {\it cost} $C_\mu(\mathcal{R})$ of an arbitrary pmp discrete equivalence relation $\mathcal{R}$ over a probability space $(X,\mu)$ thoroughly, following Levitt's former work \cite{levitt}. He made many non-trivial computations including that $C_\mu(\mathcal{R}_{\mathbb{F}_n\curvearrowright(X,\mu)}) = n$ for any essentially free, pmp action $\mathbb{F}_n \curvearrowright (X,\mu)$ possibly with $n=\infty$. He also proved, in his work \cite{gaboriau:betti} on $L^2$-Betti numbers, the following inequality   
\begin{equation}\label{b}
\beta_1^{(2)}(\mathcal{R}) - \beta_0^{(2)}(\mathcal{R}) +1 \leq C_\mu(\mathcal{R}).
\end{equation}
Gaboriau's theory of costs including this inequality also seems missing for arbitrary pmp discrete groupoids. It is rather straightforward, see \cite{ueda},\cite{abert-nikolov},\cite{abert-weiss}, to adapt Levit-Gaboriau's definition of costs to pmp discrete groupoids. However, it is certainly non-trivial to generalize the main assertions in Gaboriau's theory of costs. In fact, \cite[Proposition I.11]{gaboriau:cost} does never hold true for pmp discrete groupoids (see \cite[Remark 12 (1)]{ueda}). Nevertheless, Ueda \cite{ueda} showed that some important others, e.g.~\cite[Proposition II.6, Th{\'e}or{\`e}me IV.15]{gaboriau:cost}, still hold true for arbitrary pmp discrete groupoids, but his work was done in terms of operator algebras. In the present notes we will translate his work into terms of pmp discrete groupoids by supplying necessary technical ingredients, and then establish the formula \eqref{b} for arbitrary pmp discrete groupoids by generalizing necessary parts of Gaboriau's theory to the groupoid setting. We also compute the costs of pmp `treeable' groupoids. 

As mentioned above the present notes supply necessary explanations for unifying previous fundamental works on $L^2$-Betti numbers and costs in the class of pmp discrete groupoids. Hence some parts of the present notes may have implicitly been known so far, though nobody explored them in any literature. We intend to provide the present notes as a reference for future study of pmp discrete groupoids. We use the necessary contents from Sauer's paper \cite{sauer} without explanation and also some technical things from \cite{neshveyev} to make these notes short enough. Nevertheless, these notes with the help of only \cite{neshveyev}, \cite{sauer}, and \cite{ueda} are essentially self-contained.

\section{Pmp discrete groupoids and their von Neumann algebras} 

Let $G$ be a discrete (standard) Borel groupoid with unit space $X$ (usually denoted by $G^{(0)}$ instead). The source map and the range map are denoted by $s : G \to X$ and $r : G \to X$, respectively. If the mapping $g \in G \mapsto (r(g),s(g)) \in X\times X$ is injective, we say that $G$ is {\it principal}. In this case, $G$ is nothing but a discrete Borel equivalence relation. A Borel subset $E \subset G$ is said to be {\it one-sheeted} if $s \upharpoonright_E$ and $r \upharpoonright_E$ are injective. The symbol $\mathcal{G}_G$ denotes the set of one-sheeted sets of $G$. Since $s$ and $r$ are countable-to-one maps, the following hold true (due to e.g.~\cite[Theorems 15.1, 15.2, 18.10]{kechris:GTM156}): (i) $G$ can be decomposed into countable disjoint union of elements in $\mathcal{G}_G$; (ii) for each $E \in \mathcal{G}_G$ we have a partially defined Borel isomorphism $\varphi_E := (r\!\upharpoonright_E)\circ(s\!\upharpoonright_E)^{-1} : s(E) \to r(E)$. Assume that $X$ is endowed with a probability measure $\mu$ which is invariant under all $\varphi_E$, $E \in \mathcal{G}_G$. We call such a pair $(G,\mu)$ a {\it pmp discrete groupoid}. Define a (possibly infinite) measure $\mu^G$ on $G$ by $\mu^G(B) = \int_X\, \#(s^{-1}(\{x\})\cap B)\, \mu(dx)$ for every Borel subset $B$ of $G$. 

\medskip
The {\it groupoid ring} $\mathbb{C}[G]$ of $G$ is defined to be the linear subspace of functions $f \in L^\infty(G,\mu^G)$ such that two functions $x \mapsto \#(s^{-1}(x) \cap \supp f)$, $x \mapsto \#(r^{-1}(x) \cap \supp f)$ are bounded $\mu$-a.e. The product $(f_1, f_2) \in \mathbb{C}[G]\times\mathbb{C}[G] \mapsto f_1 f_2 \in \mathbb{C}[G]$ and the adjoint $f \in \mathbb{C}[G] \mapsto f^* \in \mathbb{C}[G]$ are defined by $(f_1 f_2)(g) = \sum_{ g_1 g_2 = g} f_1(g_1) f_2(g_2)$ and $(f^*)(g) := \overline{f(g^{-1})}$, respectively. With these operations, $\mathbb{C}[G]$ becomes a $*$-algebra. We remark that if $G$ is a discrete group, then $\mathbb{C}[G]$ is just the usual group ring.  

\medskip
The so-called {\it{\rm(}left{\rm)} regular representation} $\mathbb{C}[G] \curvearrowright L^2(G) := L^2(G,\mu^G)$ is defined by $(f \xi)(g) := \sum_{g_1 g_2 = g} f(g_1)\xi(g_2)$ for $f \in \mathbb{C} [G]$ and $\xi \in L^2(G)$, and it generates the {\it groupoid von Neumann algebra} $L(G) = \mathbb{C}[G]''$ on $L^2(G)$. The von Neumann algebra $L(G)$ has a faithful normal tracial state $\tau_G$ defined by a cyclic and separating vector $\mathbbm{1}_X$ (the characteristic function on $X$). Remark that each $u(E) := \mathbbm{1}_E$ inside $L(G)$, $E \in \mathcal{G}_G$, defines a partial isometry in $L(G)$ and that $L(G)$ is generated by these $u(E)$ as a von Neumann algebra, since $G$ is a countable disjoint union of one-sheeted sets. In closing of this section, we give two remarks: (1) If $G$ is a discrete group, then $(L(G),\tau_G)$ is nothing but the {\it group von Neumann algebra} with the canonical tracial state. (2) If $G$ is the {\it transformation groupoid} (see the glossary prior to Lemma \ref{trivial_action} for the definition) arising from a pmp action $\Gamma \curvearrowright X$ of a discrete group, then $L(G)$ is naturally identified with $L^\infty(X)\rtimes\Gamma$, the crossed product of $L^\infty(X)$ by the induced action of $\Gamma$ on $L^\infty(X)$ in the sense of e.g.~\cite[Definition 13.1.3]{kadison-ringrose:bookv2}. The identification is precisely given by $u(E_\gamma) = u_\gamma \otimes \lambda_\gamma$, where $E_\gamma := X \times \{ \gamma \}$, $u_\gamma$ is the unitary representation of $\Gamma$ on $L^2(X)$ associated with the induced action, and $\lambda_\gamma$ denotes the left regular representation. 
 
Throughout the rest of this notes $(G, \mu)$ denotes a pmp discrete groupoid wiht unit space $X$.
\section{Geometric approach to $L^2$-Betti numbers of pmp discrete groupoids} 

\subsection{Definitions} 

We adapt Gaboriau's definition of $L^2$-Betti numbers to arbitrary pmp discrete groupoids with necessary suitable modifications. This and the next subsections are rather self-contained.

\medskip
A {\it {\rm(}standard{\rm)} fiber space} over $(X,\mu)$ is defined to be a pair  which consists of a (standard) Borel space $U$ and a Borel map $\pi_U : U \to X$ with countable fibers, and it is usually denoted by $U$ for simplicity. We equip it with a natural measure $\mu_U$ on $U$ defined by $\mu_U(C) := \int_X\, \# (\pi_U^{-1}(\{x\}) \cap C)\, \mu(dx)$ for every Borel subset $C$ of $U$. Any pmp discrete groupoid $(G,\mu)$ produces two fiber spaces with its source and range maps $s, r$. A Borel subset $E$ of a standard fiber space $U$ is called a {\it Borel section} of $U$ if $\pi_U \upharpoonright_E$ is injective. Note that, by \cite[Theorem 18.10]{kechris:GTM156}, any fiber space is a countable disjoint union of its Borel sections. The {\it fiber product} of fiber spaces $U_1,\dots,U_n$ means the fiber space $U_1*\dots*U_n := \{(u_1,\dots,u_n) \in U_1\times\cdots\times U_n\,|\,\pi_{U_1}(u_1) = \cdots = \pi_{U_n}(u_n) \}$ with $\pi_{U_1*\cdots*U_n} : (u_1,\cdots,u_n) \in U_1*\cdots*U_n \mapsto \pi_{U_1}(u_1) = \cdots = \pi_{U_n}(u_n) \in X$. 

\medskip
Let $U$ be a fiber space over $(X,\mu)$. We regard $G$ as a standard fiber space with the source map $s$, and get the fiber product $G*U$. In this setup, a {\it left action} of $G$ on $U$ is defined to be a Borel map $(g,u) \in G * U \mapsto g\cdot u \in U$ satisfying the following conditions: (1) $\pi_U(g \cdot u) = r(g)$, (2) $\pi_U(u) \cdot u = u$ (where $\pi_U(u)$ is viewed as an element in $G$ since $X \subseteq G$), (3) $g \cdot (g^\prime \cdot u) = (g g^\prime) \cdot u$. We call such a fiber space with left action of $G$ a {\it {\rm(}standard left{\rm)} $G$-space}. The `groupoid product map' $(g_1,g_2) \in G*G \mapsto g_1 g_2 \in G$ is nothing but a left action of $G$ on the fiber space $r: G \to X$ so that $G$ itself is a $G$-space.  

\medskip
Let $U$ be a $G$-space. The left action of $G$ is said to be {\it essentially free} if $g \cdot u = u$ implies $g = \pi_U(u)$ for $\mu_U$-a.e.~$u$. A Borel subset $F$ of $U$ is called a {\it fundamental domain} for the action of $G$ if $\#((G\cdot u) \cap F) = 1$ holds for $\mu_U$-a.e.~$u$. Following Pichot's notion \cite{pichot} we say that a $G$-space $U$ is {\it quasi-periodic}, if the left action of $G$ is essentially free and has a fundamental domain. It is important below that $G$ itself becomes a quasi-periodic $G$-space with fundamental domain $X$. Note that if $G$ is principal or other words an equivalence relation, then any left action of $G$ must be essentially free. We may and do assume, by choosing smaller co-null subset if necessarily, that for any quasi-periodic $G$-space $U$, the $G$-action is precisely free and has an exact fundamental domain.

\medskip
The next lemma is crucial and the groupoid counterpart of \cite[Lemme 2.3]{gaboriau:betti}. 

\begin{lem} \label{lem:quasi-periodic}
 For any quasi-periodic $G$-space $U$, there exists a $G$-equivariant Borel injection from $U$ into a disjoint union $\bigsqcup_{i \in I} G = G\times I$ equipped with the left $G$-space structure as follows: its standard fiber space structure is given by the map $(g,i) \mapsto r(g)$ and its left action of $G$ is diagonal, i.e., $(g_1,(g_2,i)) \mapsto (g_1 g_2, i)$.
\end{lem} 
\begin{proof}
As we remarked above, we may assume that the action of $G$ on $U$ is exactly free and has an exact fundamental domain. Let $F$ be an exact fundamental domain for the left action of $G$ on $U$. Since $\pi_U \!\upharpoonright_F : F \to X$ is a countable to one Borel map, by \cite[Theorem 18.10]{kechris:GTM156} there exists a countable Borel partition $\{ F_i \}_{i \in I}$ of $F$ such that each $\pi_U\!\upharpoonright_{F_i}$ is injective. 
Then we have $U = G \cdot F = \bigsqcup_{i \in I} G \cdot F_i$. Indeed, the first equality follows from the fact that $F$ is an exact fundamental domain and the second is due to the freeness of the action. 
Note that, by\cite[Corollary 15.2]{kechris:GTM156}, the map $\pi_U\!\upharpoonright_{F_i} : F_i \to X_i := \pi_U(F_i)$ is a Borel isomorphism so that we have a Borel injection $G \cdot X_i \to U : g \mapsto g \cdot (\pi_U \upharpoonright_{F_i})^{-1} (s(g))$ whose image is $G \cdot F_i$. Thus, by \cite[Corollary 15.2]{kechris:GTM156}, $G \cdot F_i$ is Borel and $f_i : G \cdot F_i \to G \cdot X_i : g \cdot u \mapsto g$ is an Borel isomorphism. Therefore, the desired injection $f : U \to \bigsqcup_{i \in I} G \cdot X_i$ is defined to be $f \upharpoonright_{G \cdot F_i} := f_i$, $i \in I$. 
\end{proof}

For any fiber space $U$ over $(X,\mu)$, the symbol $\Gamma(U)$ denotes the space of Borel functions $f : U \to \mathbb{C}$ such that 
$S(f)(x) := \# (\pi_U^{-1}(\{x\}) \cap \supp (f))$ is finite for $\mu$-a.e.~$x$, where $\supp(f):=\{u \in U\,|\,f(u)\neq0\}$. We also define $\Gamma^b(U)$ to be the space of $f \in \Gamma(U) \cap L^\infty(U)$ such that $S(f) \in L^\infty(X)$, and set $\Gamma^{(2)}(U) := L^2(U,\mu_U)$. Note that every function on $U$ is the sum of functions each of which is of the form $(\xi \circ \pi_U) \mathbbm{1}_E$; here $\xi$ is a measurable function on $X$ and $E$ is a Borel section of $U$. In the following the symbol $\Gamma^\star(U)$ denotes the one of $\Gamma(U)$, $\Gamma^b(U)$ and $\Gamma^{(2)}(U)$. 

\medskip
Let $U$ be a $G$-space. Then $\Gamma^\star(U)$ have the following natural left $\mathbb{C}[G]$-module structure:
$$
(f \varphi)(u) := \sum_{g \in r^{-1}(\{\pi(u)\})} f(g)\varphi(g^{-1} \cdot u)
$$ 
for $f \in \mathbb{C}[G]$ and $\varphi \in \Gamma^\star(U)$. If $U$ is quasi-periodic, then $\Gamma^{(2)}(U)$ becomes a Hilbert $L(G)$-module whose {\it Murray-von Neumann dimension} (see \cite[\S 1.1]{luck:survey}) equals the measure of a fundamental domain of $U$. Indeed, since we may assume that $U = \bigsqcup_{i \geq 1} G \cdot X_i$ (see the proof of Lemma \ref{lem:quasi-periodic}), we have $\Gamma^{(2)}(U) = \sideset{}{^\oplus_{i\geq1}}\sum L^2(G) \mathbbm{1}_{X_i}$. Here, note that $(\xi \mathbbm{1}_{X_i}) (g) := \sum_{g_1 g_2} \xi(g_1) \mathbbm{1}_{X_i}(g_2)$ ( i.e., the right action of $\mathbbm{1}_{X_i}$), which defines the projection $\xi \mapsto \xi \mathbbm{1}_{X_i}$ in the commutant $L(G)^\prime$. Thus we conclude that $\Gamma^{(2)}(U)$ is a Hilbert $L(G)$-module and that $\dim_{L(G)} \Gamma^{(2)}(U) = \sum_{i \geq 1} \mu(X_i)$, which equals the measure of a fundamental domain.

\medskip
For a $G$-space $U$, any fiber product $U*\cdots*U$ becomes again a $G$-space endowed with the diagonal action of $G$: $(g,(u_1,\dots,u_n)) \mapsto (gu_1,\dots,gu_n)$. A {\it simplicial $G$-complex} is defined to be a sequence $\Sigma = (\Sigma^{(n)})_{n \geq 0}$ of quasi-periodic $G$-spaces such that each $\Sigma^{(n)}$ is a $G$-invariant Borel subset of the $n+1$ times fiber product of $\Sigma^{(0)}$  with the restriction to $\Sigma^{(n)}$ of the left action of $G$ on the fiber product, and moreover such that the following conditions hold: 
\begin{enumerate}
\item if $(v_0, \dots, v_n) \in \Sigma^{(n)}$, then $(v_{\sigma(0)}, \dots, v_{\sigma(n)}) \in \Sigma^{(n)}$ for any permutation $\sigma$;
\item if $(v_0, \dots, v_n) \in \Sigma^{(n)}$, then $v_0 \neq v_1$;
\item if $s = (v_0, \dots, v_n) \in \Sigma^{(n)}$, then $\partial_n^j s := (v_0, \dots, \hat{v_j}, \dots, v_n) \in \Sigma^{(n-1)}$ for every $0 \leq j \leq n$, where $\hat{v_j}$ means the removal of $v_i$ from the sequence $(v_0, \dots, v_n)$.
\end{enumerate}
Note that the maps $\partial_n^j : \Sigma^{(n)} \to \Sigma^{(n-1)}$ are measurable. The fiber of $\pi_{\Sigma^{(n)}} : \Sigma^{(n)} \to X$ at $x$ is denoted by $\Sigma^{(n)}_x$. Then,  $\Sigma_x := (\Sigma^{(n)}_x)_{n \geq 0}$ becomes a usual simplicial complex; see \cite[Chapter 3]{spanier} for usual notation on simplicial complexes. We say that $\Sigma$ is {\it contractible} if so is $\Sigma_x$ $\mu$-a.e.~$x$. Similarly, we say that $\Sigma$ is {\it connected} if so is $\Sigma_x$ $\mu$-a.e.~$x$. A simplicial $G$-complex $\Sigma$ is said to be {\it uniformly locally bounded} (ULB for short) if $\Sigma^{(0)}$ has a fundamental domain of finite measure and there exists an integer $N$ such that $\# \{ s \in \Sigma_x | v \in s \} \leq N$ holds for every $v \in \Sigma^{(0)}_x$ and for $\mu$-a.e.$x$. In the case, every $\Sigma^{(n)}$ has a fundamental domain of finite measure. Indeed, if $F$ is a fundamental domain of $\Sigma^{(0)}$, then $F^{(n)} := \{ (v_0, \dots, v_n) \in \Sigma^{(n)} \, | \, v_0 \in F \}$ is a fundamental domain of $\Sigma^{(n)}$ satisfying $\mu_{\Sigma^{(n)}}(F^{(n)}) \leq N \mu_{\Sigma^{(0)}}(F) < \infty$.

\medskip
The {\it universal $G$-complex} $E G = (E G^{(n)})_{n \geq 0}$ plays an important r\^{o}le, and thus we do give its precise definition in what follows. Set $E G^{(0)} := \bigsqcup_{i \in \mathbb{N}} G = G\times\mathbb{N}$, which becomes a $G$-space with the diagonal action, see Lemma \ref{lem:quasi-periodic}. For $n \geq 1$, define $E G^{(n)}$ to be the set of $(n + 1)$-tuples $(v_0, \dots, v_n) \in EG^{(0)} * \dots * EG^{(0)}$ whose entries are distinct. Since $G$ itself is a quasi-periodic $G$-space with fundamental domain $X$ mentioned before, $EG^{(0)}$ is again a quasi-periodic $G$-space with fundamental domain $\bigsqcup_i X$ which is of infinite measure. Hence $EG$ is a contractible, simplicial $G$-complex, but infinite dimensional and far from being ULB. 

\medskip 
Let $\Sigma$ be a simplicial $G$-complex. A {\it $G$-subcomplex} of $\Sigma$ is defined to be a simplicial $G$-complex $\Xi$ such that each $\Xi^{(n)}$ is a $G$-invariant subset of $\Sigma^{(n)}$ with the restriction to $\Xi^{(n)}$ of the original left action of $G$. A sequence $(\Xi_i)_{i \geq 1}$ of $G$-subcomplexes is called an {\it exhaustion} of $\Sigma$ if $(\Xi^{(n)}_{i, x})_{i \geq 1}$ are increasing subsets of $\Sigma^{(n)}_x$ satisfying $\bigcup_{i \geq 1} \Xi_{i, x} = \Sigma^{(n)}_x$ for $\mu$-a.e.~$x$. An exhaustion $(\Xi_i)_{i \geq 1}$ is said to be ULB if each $\Xi_i$ is ULB. We will prove the existence of ULB exhaustions for any simplicial $G$-complex in the next subsection.

\medskip
For a simplicial $G$-complex $\Sigma$, let $C_n^\star (\Sigma)$ (an analogous notation as $\Gamma^\star(\Sigma)$ before) denote the subspace of $\Gamma^\star (\Sigma^{(n)})$ which consists of functions $f : \Sigma^{(n)} \to \mathbb{C}$ satisfying $f(\sigma^{-1} u) = (\mathrm{sgn} \sigma) f(u)$ for every $u \in \Sigma^{(n)}$ and every permutation $\sigma$. For $f \in C_n^\star(\Sigma)$ and $x \in X$, let $f_x$ denote the restriction of $f$ to $\Sigma_x^{(n)}$. 

The family $\{ \partial_{n, x} \}_{x \in X}$ of boundary operators on each $\Sigma^{(n)}_x$ defines a $\mathbb{C} [G]$-module map $\partial_n : C_n(\Sigma) \to C_{n -1}(\Sigma)$ as follows: for $f \in C_n(\Sigma)$, define the function $\partial_n f : \Sigma^{(n-1)} \to \mathbb{C}$ by $(\partial_n f) (u) = \partial_{n, x}(f_x)$ for $u \in \Sigma_x^{(n)}$. Then, $\partial_n f$ is measurable. Indeed, if $f = (\xi \circ \pi_{\Sigma^{(n)}}) \mathbbm{1}_E$ is supported on a Borel section $E$ of $\Sigma^{(n)}$, then we have $\partial_n f = (\xi \circ \pi_{\Sigma^{(n)}}) \sum_{j=0}^n (-1)^j \mathbbm{1}_{\partial_n^j E}$, which is clearly measurable. Thus, we get a chain complex $C_\bullet(\Sigma)$ of $\mathbb{C} [G]$-modules. 

If $\Sigma$ is ULB, then we can extend the $\partial_n$ to a unique bounded $L(G)$-module map $\partial_n^{(2)} : C_n^{(2)}(\Sigma) \to C_{n - 1}^{(2)}(\Sigma)$. Indeed, let $N$ be a constant so that $\# \{ s \in \Sigma_x\, | \, v \in s \} \leq N$ holds for $\mu$-a.e.~$x$ and every $v \in \Sigma^{(0)}_x$. Then, using the formula $(\partial_n f)_x (t) = \sum_{j = 0}^n (-1)^j \sum_{s \in (\partial_n^j)^{-1} (t)} f(s)$ and the Cauchy-Schwarz inequality, we get an estimate $\| \partial_n f \| \leq n \sqrt{N} \| f \|$ for every $f \in C_n^{(2)}(\Sigma) \cap C_n(\Sigma)$. Thus, we get a Hilbert $L(G)$-chain complex $C_\bullet^{(2)}(\Sigma)$; see \cite[\S 1.1]{luck:survey} for the terminology of Hilbert chain complexes.  

\medskip
We are ready to give the definition of $L^2$-Betti numbers of a simplicial $G$-complex. 

\begin{df}\label{geometric-def}
For a ULB simplicial $G$-complex $\Sigma$, define the {\it $n$-th reduced $L^2$-homology} of $\Sigma$ by 
\begin{equation}\label{c}
\overline{H}_n^{(2)}(\Sigma, G) 
:= H_n^{(2)}(C_\bullet^{(2)}(\Sigma))
 =  \ker\partial_n^{(2)}/\,\overline{\mathrm{im}\partial_{n + 1}^{(2)}}. 
\end{equation}
Here notice that $\overline{H}_n^{(2)}(\Sigma, G)$ becomes a Hilbert space, since we have taken the closure of $\mathrm{im}\partial_{n+1}^{(2)}$.

For an arbitrary simplicial $G$-complex $\Sigma$, we take a ULB {\it exhaustion} $\{ \Sigma_i \}_{i \geq 1}$ (possiblly with all $\Sigma_i = \Sigma$). Remark here that, for every $i \leq j$, the inclusion $\Sigma_i \subset \Sigma_j$ induces the natural bounded $L(G)$-map $J^{i,j}_{n} : C_n^{(2)}(\Sigma_i) \to C_n(\Sigma_j)$ for every $n \geq 0$ in the following manner: $J^{i, j}_n(f)(u)$ is defined to be $f(u)$ if $u \in \Sigma_i^{(n)}$ and $0$ otherwise. The maps $J^{i, j}_n$ commute with the boundary maps $\partial_n^{(2)}$, that is, $J^{i, j}_\bullet$ is a chain morphism from $C_\bullet^{(2)}(\Sigma_i)$ to $C_\bullet^{(2)}(\Sigma_j)$. Let $H_n^{(2)}(J^{i, j}_\bullet) : H_n^{(2)}(C_\bullet^{(2)}(\Sigma_i)) \to H_n^{(2)}(C_\bullet^{(2)}(\Sigma_j))$ be the natural map induced from the chain morphism $J^{i, j}_\bullet$. 
With $\nabla_n (\Sigma_i, \Sigma_j) := \dim_{L(G)} \overline{\im H_n^{(2)}(J^{i, j}_\bullet) }$, we define the {\it $n$-th $L^2$-Betti number} of $\Sigma$ by 
\begin{equation}\label{d}
\beta_n^{(2)}(\Sigma,\{\Sigma_i\}_{i\geq1},G) = \lim_{i \geq 1} \lim_{j \geq i} \nabla_n (\Sigma_i, \Sigma_j).
\end{equation} 
\end{df}

\begin{rem}\label{rem:double-limit} Let $\{ \Sigma_i \}_{i \geq 1}$ be an increasing sequence of ULB simplicial $G$-complex. Then, the function $\nabla_n (\Sigma_i, \Sigma_j)$ is increasing in $i$ and decreasing in $j$. In particular, the double limit in \eqref{d} exists. 
\end{rem}
\begin{proof}
Take $i \leq j \leq k$ arbitrary. Since the maps $H_n^{(2)}(J^{i, j}_\bullet)$ are induced from inclusion, the equality $H_n^{(2)}(J^{i, k}_\bullet) = H_n^{(2)}(J^{j, k}_\bullet) \circ H_n^{(2)}(J^{i, j}_\bullet)$ holds. Thus the map $H_n^{(2)}(J^{j, k}_\bullet)$ is a surjection from $\overline{\im H_n^{(2)}(J^{i, j}_\bullet)}$ to $\overline{\im H_n^{(2)}(J^{i, k}_\bullet)}$. Hence, by the additivity of von Neumann dimension (see \cite[Theorem 1.12 (3)]{luck:survey}), we have $\nabla_n(\Sigma_i, \Sigma_j) \geq \nabla_n(\Sigma_i, \Sigma_k)$. 
\end{proof}

It is not clear at all whether or not  the above definition of $\beta_n^{(2)}(\Sigma,\{\Sigma_i\}_{i\geq1},G)$ is independent of the choice of ULB-exhausion $\{\Sigma_i\}_{i\geq1}$. This issue will be resolved (see Proposition \ref{mainthm1}) in the course of proving the equivalence between the algebraic and the geometric approaches in \S\S 3.3.  

\subsection{A construction of ULB exhaustions} 
We prove the following proposition: 

\begin{prop}\label{exhaustion} The universal $G$-complex $EG$ has a ULB exhaustion, and hence so does any $G$-complex. 
\end{prop}

For every $N \geq 1$, define the $G$-subcomplex $(EG)_N$ of $E G$ in the following manner: Set $(EG)_N^{(0)} = \bigsqcup_{i=1}^N G = G\times\{1,\dots,N\}$ that naturally sits in $EG^{(0)}$. For $n \geq 1$, define $(EG)_N^{(n)}$ to be the set of $(n + 1)$-tuples $(v_0, \dots, v_n) \in (EG)_N^{(0)} * \cdots * (EG)_N^{(0)}$ whose entries are distinct. 

\begin{lem} The $G$-complex $(EG)_N$ has a ULB exhaustion for every $N \geq 1$.
\end{lem}
\begin{proof} 
Fix $N \geq 1$. Let $G = \bigsqcup_{i \geq 1}  E_i$ be a decomposition into a countable family of one-sheeted sets; see \S2. For $k \geq 1$, we set $\tilde{E}_k = \bigsqcup_{i = 1}^k E_i$ and define $\Sigma_k = (\Sigma_k^{(n)})_{n \geq 0}$ in the following way: set $\Sigma_k^{(0)} := (EG)_N^{(0)}$; for $n \geq 1$ let $\Sigma_k^{(n)}$ be the set of $((g_0,i_0),\dots,(g_n,i_n)) \in (EG)_N^{(n)}$ such that $g_{j + 1}^{-1} g_j \in \tilde{E}_k \tilde{E}_k^{-1}$ holds for every $0 \leq j \leq n - 1$. 

We show that the sequence $(\Sigma_k)_{k \geq 1}$ is a ULB exhaustion of $(EG)_N$. 

Remark that, if $( (g_0, i_0), \dots, (g_n, i_n) ) \in (\Sigma_k)^{(n)}$, then $g_j^{-1} g_{j^\prime} \in \tilde{E}_k \tilde{E}_k^{-1}$ holds for every $j \neq j^\prime$. Indeed, by the definition of $(\Sigma_k)^{(n)}$, there exist $h_0, \dots, h_n \in \tilde{E}_k$ so that $g_{j + 1}^{-1} g_j = h_{j + 1} h_j^{-1}$ holds for every $0 \leq j \leq n - 1$. Thus, for $0 \leq j < j^\prime \leq n$, we have $g_{j^\prime}^{-1} g_j = g_{j^\prime} g_{j^\prime - 1} g_{j^\prime - 1}^{-1} g_{j^\prime - 2} \cdots g_{j + 1}^{-1} g_j = h_{j^\prime} h_{j^\prime - 1}^{-1} h_{j^\prime - 1} h_{j^\prime - 2}^{-1} \cdots h_{j + 1} h_j^{-1} = h_{j^\prime} h_j^{-1} \in \tilde{E}_k \tilde{E}_k^{-1}$. We also have $g_j^{-1} g_{j^\prime} = (g_{j^\prime}^{-1} g_j )^{-1} = h_j h_{j^\prime}^{-1} \in \tilde{E}_k \tilde{E}_k^{-1}$ by taking their inverses.

In what follows, we divide the proof into three steps.

\medskip
({\bf Step 1}: Each $\Sigma_k$ is a $G$-subcomplex of $(EG)_N$.) For any $g \in G$, we have $g \cdot s = ((g g_0, i_0), \dots, (g g_n, i_n))$ and $(g g_{j + 1})^{-1} (g g_j) = g_{j + 1}^{-1} g_j \in \tilde{E}_k \tilde{E}_k^{-1}$ for every $0 \leq j \leq n - 1$. Thus, each $\Sigma_k^{(n)}$ is a $G$-invariant subset of $(EG)_N^{(n)}$. Also, by the above remark, each $\Sigma_k$ is clearly a simplicial $G$-complex. Thus $\Sigma_k$ is a $G$-subcomplex.

\medskip
({\bf Step 2}: Each $\Sigma_k$ is ULB.) Take $x \in X$ and $(g_0, i_0) \in (\Sigma_k)_x^{(0)}$. We show that the number of elements $s \in (\Sigma_k)_x$ containing $(g_0, i_0)$ as the first component is not larger  than a universal constant ( i.e., it is independent of the choice of $x$ and $(g_0, i_0)$). 

Choose $s = ((g_0, i_0), \dots, (g_n, i_n)) \in (\Sigma_k)_x^{(n)}$. Then, by the definition of $\Sigma_k^{(n)}$, there exist $h_0, \dots, h_n \in \tilde{E}_k$ so that $g_j^{-1} g_{j^\prime} = h_j h_{j^\prime}^{-1}$ for every $j \neq j^\prime$. Thus, $h_j = h_{j^\prime}$ implies that $g_j^{-1} g_{j^\prime}$ falls in the unit space, and hence $g_j = g_{j^\prime}$, a contradiction by the definition of $(EG)_N^{(n)}$. Therefore, $h_0, \dots, h_n$ must be different. Also, we have $g_j = g_0 g_0^{-1} g_j = g_0 h_0 h_j^{-1}$ for every $0 \leq j \leq n$. 

Define $H_{n, x, g_0}$ to be the set of $(h_0, \dots, h_n) \in \tilde{E}_k \times \cdots \times \tilde{E}_k$ satisfying the following conditions: (1) $h_0, \dots, h_n$ are different; (2) $r(h_0) = s(g_0)$; (3) $s(h_j) = s(h_0)$ for every $0 \leq j \leq n$. Then, by what we have proved in the previous paragraph, the image of the map 
$$
H_{n, x, g_0} \times \{ 1, \dots, N \}^n \to (\Sigma_k)_x^{(n)} : ( ( h_0, \dots, h_n), (i_1, \dots, i_n)) \mapsto ( (g_0 h_0 h_j^{-1}, i_j)_{j=0}^n )
$$
is equal to $\{ s \in (\Sigma_k)_x^{(n)}\, | \, (g_0, i_0) \in s \}$. Therefore, we have $\# \{ s \in (\Sigma_k)_x \, | \, (g_0, i_0) \in s\} \leq \sum_{n = 0}^\infty N^n \times \# H_{n, x, g_0}$. 

We give an estimate of $ \# H_{n, x, g_0}$ from the above. Take $(h_0, \dots, h_n) \in H_{n, x, g_0}$. By the definition of  $H_{n, x, g_0}$, we see that $h_0 \in \bigsqcup_{j = 1}^k E_j \cap r^{-1}(s( g_0))$.Since each $E_j$ is a one-sheeted set, we have $\# (E_j \cap r^{-1}(s(g_0))) \leq 1$ for every $1 \leq j \leq k$. Thus, the number of choice for $h_0$ is not larger than $k$. Without loss of generality, we may assume that $h_0 \in E_1 \cap r^{-1}(s(g_0))$. Then, by the definition of $H_{n, x, g_0}$, we have $h_1, \dots, h_n \in \bigsqcup_{j = 2}^k E_j \cap s^{-1}(s(h_0))$. Let $j_l$ denote the index so that $h_l \in E_{j_l} \cap s^{-1}(s(h_0))$ for every $1 \leq l \leq n$. Then, since $h_1, \dots, h_n$ are different and each $E_j$ is one-sheeted, $j_1, \dots, j_n$ must be different. Since $\# (E_j \cap s^{-1}(s(h_0))) \leq 1$ for every $2\leq j \leq k$, the number of choices for $(h_1, \dots, h_n)$ is not larger than the number of sequences $(j_1, \dots, j_n)$ which consists of different elements of $\{ 2, \dots, k\}$. Hence, $\# H_{n, x, g_0} \leq k (k -1) \cdots (k - n)$ if $n \leq k - 1$. Clearly $H_{n , x, g_0} = \emptyset$ if $n \geq k$. 

Therefore, we conclude that $\# \{ s \in (\Sigma_k)_x \, | \, (g_0, i_0) \in s\} \leq \sum_{n=0}^{k - 1} N^n k (k -1) \cdots (k - n)$, which is independent of the choice of $(x, g_0)$. 

Let us show that $\Sigma_k^{(0)} = (EG)_N^{(0)}$ has a fundamental domain of finite measure. Note that $F_N := \bigsqcup_{i=1}^N X = X \times \{ 1, \dots, N \}$ is a fundamental domain of $\Sigma_k^{(0)} = (EG)_N^{(0)}$. Since $\# ((EG)_{N, x}^{(0)} \cap F_N) = N $ for every $x$, we have $\mu_{\Sigma_k^{(0)}}(F_N) = N < \infty$. 

\medskip
({\bf Step 3}: The sequence $(\Sigma_k)_{k \geq 1}$ is an exhaustion of $(EG)_N$.) It is clear that each $(\Sigma_k^{(n)})_k$ is increasing by definition. It suffices to show that $((EG)_N)_x^{(n)} = \bigcup_{k\geq 1} (\Sigma_k)_x^{(n)}$ holds for every $n \geq 0$ and $x \in X$. Take $x \in X$, $n \geq 0$ and $s = ((g_0, i_0), \dots, (g_n, i_n)) \in (EG)_{N, x}^{(n)}$. Since $G = \bigsqcup_{k\geq 1} E_k$, we have $g_0^{\pm 1}, \dots, g_n^{\pm 1} \in \tilde{E}_j$ for some $j \geq 1$ so that $s \in \Sigma_{j, x}^{(n)}$. Hence we are done. 
\end{proof}

We are ready to prove Proposition \ref{exhaustion}.

\begin{proof} (Proposition \ref{exhaustion}) 
Let $(\Sigma_{N,k})_{k \geq1}$ be a ULB-exhaustion of $(EG)_N$ for each $N \geq 1$, whose existence was established by the above lemma. Then, the sequence $(\Sigma_{k,k})_{k \geq 1}$ is clearly a ULB-exhaustion of $EG$. Note also that any simplicial $G$-complex $\Xi$ can be embedded $G$-equivariantly into the universal $G$-complex $E G$ thanks to Lemma \ref{lem:quasi-periodic}. Then, the sequence $(\Sigma_k)_{k \geq 1}$ defined by $\Sigma_k^{(n)} := \Sigma_{k,k}^{(n)} \cap \Xi^{(n)}$, $n \geq 0$, becomes a ULB-exhaustion of $\Xi$. 
\end{proof}

\subsection{Justification}

We will justify the geometric definition of $L^2$-Betti numbers of pmp discrete groupoids following the idea of Neshveyev and Rustad \cite{neshveyev} (that seems to originate in \cite[Remark 6.76]{luck:survey} dealing with the discrete group case). In what follows, we use L\"{u}ck's extention of the usual Murray-von Neumann dimension to arbitrary modules (see \cite{luck_98},\cite[\S\S 6.1]{luck:survey}) with keeping the same symbol $\dim_M$.  

The next theorem is the main result of this section. Recall that Sauer \cite{sauer} defined the ($n$-th) $L^2$-Betti number of $G$ by $\beta_n^{(2)}(G) = \dim_{L(G)} \Tor_n^{\mathbb{C} [G]} (L(G), L^\infty(X))$. 

\begin{thm} \label{mainthm2}
If $\Sigma$ is a contractible, simplicial $G$-complex, then $\beta_n^{(2)}(\Sigma,G) = \beta_n^{(2)}(G)$ holds for every $n \geq 0$.
\end{thm}

First, we prove the following proposition: 

\begin{prop}\label{mainthm1}
For any simplicial $G$-complex $\Sigma$ and any  ULB {\it exhaustion} $\{ \Sigma_i \}_{i \geq 1}$ of $\Sigma$, we have 
\begin{align*}
\beta_n^{(2)}(\Sigma,\{\Sigma_i\}_{i\geq1},G) 
= \dim_{L(G)} H_n(L(G) \otimes_{\mathbb{C}[G]} C_\bullet^b(\Sigma))
= \dim_{L(G)} H_n(L(G) \otimes_{\mathbb{C}[G]} C_\bullet(\Sigma))
\end{align*}
for every $n \geq 0$. In particular, $\beta_n^{(2)}(\Sigma,\{\Sigma_i\},G)$ is independent of the choice of $\{\Sigma_i\}_{i\geq1}$ so that we write $\beta^{(2)}(\Sigma,G) := \beta_n^{(2)}(\Sigma,\{\Sigma_i\},G)$ from now on. 
\end{prop}

Before proving the proposition, we provide a terminology and some general lemmas. Let $(M, \tau)$ be a finite von Neumann algebra equipped with a faithful normal tracial state. A  morphism $h : Q_1 \to Q_2$ between two $M$-modules is called a {\it $\dim_M$-isomorphism } if both $\dim_M \ker h$ and $\dim_M \mathrm{coker}\, h$ is zero. In the case, $\dim_M(Q_1) = \dim_M(Q_2)$ holds thanks to the additivity of $\dim_M$ (see \cite[Theorem 6.7 (4) (b)]{luck:survey}). See e.g.~\cite[\S2]{sauer} for further nice properties on $\dim_M$-isomorphisms. For an $M$-module $Q$, the {\it rank norm} $[ \xi ]_M$ of $\xi \in Q$ is defined to be $\inf \{ \tau(p) \, | \, p \in M^p, p \xi = \xi \}$. Then $d_M ( \xi, \eta) := [ \xi - \eta ]_M$ defines a pseudo metric on $Q$. The procedure of completion in the metric $d_M$ defines a functor $c_M$, called the {\it functor of rank completion}, from the category of $M$-modules to itself. See \cite[\S2]{thom} and \cite[Lemma 1.1]{neshveyev} for more on this functor $c_M$ and its connection with the dimension function $\dim_M$. 

Here we quote two general lemmas from \cite{neshveyev}.

\begin{lem}\label{lem:alg1} (\cite[Lemma 1.3]{neshveyev}) 
Let $N \subset \mathfrak{M} \subset M$ be a triple of algebras such that $N$ and $M$ are finite von Neumann algebras with faithful normal tracial states $\tau_N$ and $\tau_M$, respectively. Assume that the inclusion $N \subset \mathfrak{M}$ satisfies the following condition: for any $m \in \mathfrak{M}$ and $\epsilon > 0$, there exists a $\delta > 0$ such that if $p \in N^p$ satisfies $\tau_N (p)  < \delta$, then $[ m p ]_N < \epsilon$. Then, for any $\dim_N$-isomorphic $\mathfrak{M}$-map $Q_1 \to Q_2$, the induced $\mathfrak{M}$-map $\mathrm{Tor}_n^{\mathfrak{M}} (M, Q_1) \to \mathrm{Tor}_n^{\mathfrak{M}} (M, Q_2)$ is $\dim_M$-isomorphic for every $n \geq 0$.
\end{lem}

\begin{lem}\label{lem:alg2}(\cite[Lemma 1.4]{neshveyev})
Let $N \subset \mathfrak{M} \subset M$ be as in Lemma \ref{lem:alg1}. 
Assume that the pair $N \subset \mathfrak{M}$ satisfies the assumption of Lemma \ref{lem:alg1}. Then, for any resolution $P_\bullet$ of an $\mathfrak{M}$-module $Q$ such that each $P_k$ has a $d_N$-dense projective submodule, we have $\dim_M \mathrm{Tor}_n^\mathfrak{M} (M, Q) = \dim_M H_n(M \otimes_\mathfrak{M} P_\bullet)$ for every $n \geq 0$.
\end{lem}

In order to use the above lemmas in our situation, we prove the next two lemmas. 

\begin{lem} \label{lem:alg3} The pair $L^\infty(X) \subset \mathbb{C}[G]$ satisfies the assumption of Lemma \ref{lem:alg1}. 
\end{lem}

\begin{proof} 
It is known, see \cite[Lemma 3.3]{sauer}, that any element in $\mathbb{C}[G]$ is written as a finite sum of elements in $\mathbb{C}[G]$ supported in one-sheeted sets. Hence it suffices to show that $[f \mathbbm{1}_Z] \leq \tau(\mathbbm{1}_Z)$ for every $f \in \mathbb{C}[G]$ supported in a one-sheeted set $E$ and every subset $Z$ of $X$. We have $(f \mathbbm{1}_Z) (g) = f(g) \mathbbm{1}_E (g) \mathbbm{1}_Z (r(g)) = f(g) \mathbbm{1}_E (g) \mathbbm{1}_{\varphi_E^{-1}(Z)}(s(g)) = (\mathbbm{1}_{\varphi_E^{-1}(Z)} f )(g)$ for all $g \in G$. Hence we have $[ f \mathbbm{1}_Z ]_{L^\infty(X)} = [ \mathbbm{1}_{\varphi_E^{-1}(Z)} f ]_{L^\infty(X)} \leq \tau (\mathbbm{1}_{\varphi_E^{-1}(Z)}) = \mu(\varphi_E^{-1}(Z)) \leq \mu(Z) = \tau (\mathbbm{1}_Z)$. Here the first inequality simply follows from the definition of the rank norm and the second one from the fact that $\varphi_E$ is $\mu$-preserving.
\end{proof}

\begin{lem} \label{lemformain1}
Let $U$ be a quasi-periodic $G$-space with fundamental domain $F$. Then, 
\begin{enumerate}
\item $\Gamma(U)$ has a $d_{L^\infty(X)}$-dense, projective $\mathbb{C}[G]$-submodule;
\item if $\mu_U(F) < \infty$, then the $\mathbb{C}[G]$-map 
$h : L(G) \otimes_{\mathbb{C}[G]} \Gamma^b(U) \to \Gamma^{(2)}(U)$ sending $m \otimes \xi$ to $m \cdot \xi$ is a $\dim_{L(G)}$-isomorphism.
\end{enumerate}
\end{lem}
\begin{proof}
By Lemma \ref{lem:quasi-periodic} (or more precisely its proof), we may assume that $U = \bigsqcup_{i=1}^\infty G \cdot X_i$. Consider the projective $\mathbb{C}[G]$-module $P := \bigoplus_{i \geq 1} \mathbb{C}[G]\,\mathbbm{1}_{X_i}$ sitting inside $\Gamma^b(U)$.

(1) Take $f \in \Gamma(U)$. For each $m \geq 1$, define $Y_m := \{ x \in X\, | \, \supp f \cap \pi_U^{-1}(x) \subset \bigsqcup_{i=1}^m G \cdot X_i \}$. Then $\{ Y_m \}_m$ is an increasing sequence satisfying $\mu(X \setminus \bigcup_{m = 1}^\infty Y_m) = 0$, and hence $d_{L^\infty(X)}( \mathbbm{1}_{Y_m} f, f) \leq \mu(Y_m^c) \to 0$ as $m \to \infty$. Note that $\mathbbm{1}_{Y_m} f$ is supported in $\bigsqcup_{i = 1}^m G \cdot X_i$. Hence $P$ is $d_{L^\infty(X)}$-dense in $\Gamma(U)$, because so is $\mathbb{C}[G]$ in $\Gamma(G)$ as shown below. Take $f \in \Gamma(G)$. Let us decompose $G$ into one-sheeted sets $G = \bigsqcup_{i = 1}^\infty E_i$; see \S2. For each $m \geq 1$, define $Z_m$ to be the set of $x \in X$ satisfying $\sup_{g \in s^{-1}(x)} | f(g) | \leq m$ and $(\supp f \cap s^{-1}(x)) \subset \bigcup_{i=1}^m E_i \cap s^{-1}(x)$. Clearly, $\mathbbm{1}_{Z_m} f \in \mathbb{C}[G]$ converges to $f$ in $d_{L^\infty(X)}$. Consequently, we have seen that $P$ is a desired projective $\mathbb{C} [G]$-module.

(2) We have $\Gamma^{(2)}(U) = \sideset{}{^\oplus_{i \geq1}} \sum L^2(G) \mathbbm{1}_{X_i}$, see \S\S 3.1. With $L(G)\otimes_{\mathbb{C}[G]} P = \bigoplus_{i\geq1} L(G)\mathbbm{1}_{X_i}$ naturally, the restriction $\tilde{h}$ of $h$ to $L(G)\otimes_{\mathbb{C}[G]}P$ is exactly the inclusion $ \bigoplus_{i \geq 1} L(G) \mathbbm{1}_{X_i} \hookrightarrow \Gamma^{(2)}(U)$. Thanks to the $d_{L(G)}$-density of $L(G)$ in $L^2(G)$ together with $\sum_{i=1}^\infty \mu(X_i) = \mu_U(F) < +\infty$, it is plain to see that $\bigoplus_{i\geq1}L(G)\mathbbm{1}_{X_i}$ is $d_{L(G)}$-dense in $\sum_{i\geq1}^\oplus L^2(G)\mathbbm{1}_{X_i}$ so that $\tilde{h}$ is a $\dim_{L(G)}$-isomorphism. Since $P$ is $d_{L^\infty(X)}$-dense in $\Gamma^b(U)$ as we actually saw in the above (1), the inclusion $P \hookrightarrow \Gamma^b(U)$ is $\dim_{L^\infty(X)}$-isomorphic, and hence so is $L(G) \otimes_{\mathbb{C}[G]} P \hookrightarrow L(G) \otimes_{\mathbb{C}[G]} \Gamma^b(U)$ by Lemma \ref{lem:alg1}. Therefore, by applying the functor $c_{L(G)}$ to $\tilde{h}$ we conclude that $h$ is a $\dim_{L(G)}$-isomorphism. 
\end{proof}

Since $C_n^\star(\Sigma)$ is defined as a subspace of $\Gamma^\star(\Sigma^{(n)})$, we need the following lemma. 

\begin{lem}\label{lem:auxiliary}
Let $\Sigma$ be a simplicial $G$-complex. Then,
\begin{enumerate}
\item every $C_n(\Sigma)$ has a $d_{L^\infty(X)}$-dense projective $\mathbb{C} [G]$-submodule;
\item if $\Sigma$ is ULB, then the $\mathbb{C} [G]$-map $L(G) \otimes_{\mathbb{C} [G]} C_n^b(\Sigma) \to C_n^{(2)}(\Sigma)$ sending $m \otimes \xi$ to $m \cdot \xi$ is a $\dim_{L(G)}$-isomorphism for every $n \geq 0$. 
\end{enumerate}
\end{lem}

\begin{proof}
For a given function $f : \Sigma^{(n)} \to \mathbb{C}$, define the function $A_n f$ on $\Sigma^{(n)}$ by $(A_n f)(u) = ((n + 1) !)^{-1} \sum_{\sigma \in \mathfrak{S}_{n+1}} (\mathrm{sgn} \sigma) f(\sigma^{-1} u)$. Clearly, $A_n$ defines a $\mathbb{C}[G]$-module map $\Gamma^\star(\Sigma^{(n)})$ to $C_n^\star(\Sigma)$ that acts on $C_n^\star(\Sigma)$ trivially. 

(1) By Lemma \ref{lemformain1}, $\Gamma(\Sigma^{(n)})$ has a $d_{L^\infty(X)}$-dense projective $\mathbb{C}[G]$-submodule $P$. Therefore, $A_n(P)$ is a desired $d_{L^\infty(X)}$-dense projective $\mathbb{C}[G]$-submodule of $C_n(\Sigma)$ since $A_n$ acts $C_n(\Sigma)$ trivially and is contractive in $d_{L^\infty(X)}$. 

(2) It is plain to see that $\id \otimes A_n : L(G) \otimes_{\mathbb{C}[G]} \Gamma^b(\Sigma^{(n)}) \to L(G) \otimes_{\mathbb{C}[G]} C_n^b(\Sigma)$ is an $L(G)$-module map that acts on $L(G) \otimes_{\mathbb{C}[G]} C_n^b(\Sigma)$ trivially. Thus, applying the functor $c_{L(G)}$ and using Lemma \ref{lemformain1}, we conclude that the map $L(G) \otimes_{\mathbb{C}[G]} C_n^b(\Sigma) \to C_n^{(2)}(\Sigma)$ is a $\dim_{L(G)}$-isomorphism.   
\end{proof}

Note that since $C_n^{(2)}(\Sigma)$ is the image of the projection $A_n$, we have $\dim_{L(G)} C_n^{(2)}(\Sigma) = ((n + 1) !)^{-1} \mu_{\Sigma^{(n)}}(G \backslash \Sigma^{(n)} )$; here $G \backslash \Sigma^{(n)}$ denotes a fundamental domain of $\Sigma^{(n)}$. In particular, if $\Sigma$ is ULB, then $\dim_{L(G)} C_n^{(2)}(\Sigma)$ is finite for every $n \geq 0$.

Here is the proof of Proposition \ref{mainthm1}.

\begin{proof} (Proposition \ref{mainthm1})
First, consider the case when $\Sigma$ is ULB. The $\im \partial_{n+1}^{(2)}$ and its closure have the same $M$-dimension since the $\partial_{n+1}^{(2)} {\partial_{n+1}^{(2)}}^\ast$ maps $\overline{\im \partial_{n+1}^{(2)}}$ to $\im \partial_{n+1}^{(2)}$ injectively. Thus, one can see that the canonical surjection $q : H_n(C_\bullet^{(2)}(\Sigma)) \to \overline{H}_n^{(2)}(\Sigma,G)$ is a $\dim_{L(G)}$-isomorphism. Since $\Sigma$ is ULB, Lemma \ref{lem:auxiliary} enables us to obtain a $\dim_{L(G)}$-isomorphism $h : L(G) \otimes_{\mathbb{C}[G]} C_n^b(\Sigma) \to C_n^{(2)}(\Sigma)$ so that the induced $L(G)$-map $h_* : H_n(L(G)\otimes_{\mathbb{C}[G]} C_\bullet^b(\Sigma)) \to H_n(C_\bullet^{(2)}(\Sigma))$ is a $\dim_{L(G)}$-isomorphism for every $n \geq 0$. Thus, $q \circ h_* : H_n(L(G) \otimes_{\mathbb{C}[G]}C_\bullet^b(\Sigma)) \to \overline{H}_n^{(2)}(\Sigma,G)$ is a $\dim_{L(G)}$-isomorphism for every $n \geq 0$.

Next, consider the case when $\Sigma$ is an arbitrary simplicial $G$-complex. Let $\{ \Sigma_i \}_{i \geq 1}$ be a ULB-exhaustion of $\Sigma$. By what we have actually proved in the previous paragraph, together with the continuity of $\dim_{L(G)}$ under inductive limit (\cite[Theorem 6.13]{luck:survey}), we have $\beta_n^{(2)}(\Sigma,\{\Sigma_i\}_{i \geq 1}, G) = \dim_{L(G)} H_n(L(G)\otimes_{\mathbb{C}[G]} \bigcup_{i \geq 1} C_\bullet^b(\Sigma_i))$. Since $\bigcup_{i \geq 1}C_n^b(\Sigma_i)$ is $d_{L^\infty(X)}$-dense in $C_n^b(\Sigma)$, Lemma \ref{lem:alg1} shows that the last quantity equals $\dim_{L(G)} H_n(L(G) \otimes_{\mathbb{C}[G]} C_\bullet^b(\Sigma))$. Hence the proof of the first equality is completed. 

The second equality immediately follows from the $d_{L^\infty(X)}$-density of $C_\bullet^b(\Sigma)$ in $C_\bullet(\Sigma)$ and Lemma \ref{lem:alg1}. 
\end{proof}
 
We prove Theorem \ref{mainthm2} using Proposition \ref{mainthm1}. This will be done by showing the exactness of the chain complex $\dots \stackrel{\partial_2}{\rightarrow} C_1(\Sigma) \stackrel{\partial_1}{\rightarrow} C_0(\Sigma) \stackrel{\epsilon}{\rightarrow} M(X) \to 0$ of $\mathbb{C}[G]$-modules for a contractible, simplicial $G$-complex $\Sigma$; 
here $M(X)$ denotes the space of measurable functions on $X$ and $\epsilon$ denotes the $\mathbb{C} [G]$-module map defined by $\epsilon(f)(u) := \sum_{u \in \Sigma_x^{(0)}} f(u)$. 

To this end, we provide a terminology and lemmas. Let $V$ be a vector space over $\mathbb{Q}$ of countable dimension. We endow $V$ with the discrete Borel structure. A family $\{ V_x \}_{x \in X}$ of subspaces of $V$ is said to be {\it measurable} if for any measurable map $s : X \to V$, the set $\{ x \in X \, | \, s(x) \in V_x \}$ is measurable. A family $\{ T_x \}_{x \in X}$ of ($\mathbb{Q}$-linear) operators on $V$ is said to be {\it measurable} if for any measurable map $s : X \to V$, the map $X \ni x \mapsto T_x s(x) \in V$ is measurable. We can check that the measurability of a family $\{V_x\}_{x \in X}$ (resp. $\{T_x\}_{x \in X}$) is equivalent to that of the map $X \ni x \mapsto V_x \in 2^V$ (resp. $X \ni x \mapsto T_x \in V^V$). We quote two lemmas from \cite{neshveyev}.

\begin{lem}\label{lem:proj}(\cite[Lemma 2.4]{neshveyev})
If $\{V_x\}_{x \in X}$ is a measurable family of subspaces of $V$, then there exists a measurable family $\{ p_x \}_{x \in X}$ of projections onto $V_x$.
\end{lem}

\begin{lem}\label{lem:homotopy}(\cite[Lemma 2.5]{neshveyev})
Let $\{T_x \}_{x \in X}$, $\{p_x \}_{x \in X}$ and $\{q_x \}_{x \in X}$ are measurable families of operators on $V$ such that the $p_x$ and the $q_x$ are projections. Assume that, for every $x \in X$, the map $T_x$ maps $\ker q_x$ to $\im p_x$ bijectively. Let $S_x$ denotes the operator on $V = \ker p_x \bigoplus \im p_x$ defined by $S_x \upharpoonright_{\ker p_x} = 0$ and $S_x \upharpoonright_{\im p_x} = (T_x \upharpoonright_{\ker q_x})^{-1}$, so that $T_x S_x = p_x$ and $S_x T_x = \id_V - q_x$. Then the family $\{ S_x \}_{x \in X}$ is measurable. 
\end{lem}

The next lemma is just a translation of \cite[Proposition 2.6]{neshveyev} into our situation. However, we do give its proof for the sake of completeness.

\begin{lem} \label{resolution}
Let $\Sigma$ be a contractible, simplicial $G$-complex. Then the sequence 
$$
\dots \stackrel{\partial_2}{\rightarrow} C_1(\Sigma) \stackrel{\partial_1}{\rightarrow} C_0(\Sigma) \stackrel{\epsilon}{\rightarrow} M(X) \to 0
$$
is contractible as a chain complex of $L^\infty(X)$-modules.
\end{lem}

\begin{proof}
First, we consider the same sequence with rational coefficients. Let $V$ be the vector space which consists of finitely supported functions $f : \mathbb{N} \to \mathbb{Q}$. Clearly, $V$  is of countable dimension. Construct an embedding $C_n(\Sigma_x ; \mathbb{Q}) \to V$ for each $n \geq 0$ as follows: since $\Sigma^{(n)}$ can be written as a disjoint union of its Borel sections, we may regard $\Sigma^{(n)}$ as a fiber subspace of the trivial fiber space $X \times \mathbb{N}$. Then, each $\Sigma_x^{(n)}$ is a subset of $\{ x \} \times \mathbb{N}$. Thus, we can regard $C_n(\Sigma_x; \mathbb{Q})$ as $V$ naturally. It is not hard to see that $x \mapsto \ker \partial_{n, x} \subset C_n(\Sigma_x; \mathbb{Q})$ is measurable. Hence, applying Lemma \ref{lem:proj}, we get a measurable family $\{p_{n, x} \}_{x \in X}$ of projections onto $\ker \partial_{n, x}$. The contractibility of $\Sigma$ gurantees that $\partial_{n, x}$ maps $\ker p_{n +1, x}$ to $\im p_{n, x}$ bijectively. Thus, applying Lemma \ref{lem:homotopy}, we obtain measurable families $\{ h_{n, x} \}_{x \in X}$ ($n \geq -1$) of operators $h_{n, x} : C_n(\Sigma_x; \mathbb{Q}) \to C_{n + 1}(\Sigma_x; \mathbb{Q})$ satisfying 
\begin{equation}\label{h}
\mathrm{id} = h_{n - 1, x} \circ \partial_{n, x} + \partial_{n+1, x} \circ h_{n, x}
\end{equation} 
for every $n \geq -1$ (with $C_{-1}(\Sigma_x; \mathbb{Q}) = \mathbb{Q}$, $\partial_{0, x} = \epsilon_x$).

Next, consider the sequence with complex coefficients. By linearity we extend each $h_{n, x}$ to an operator from $C_n(\Sigma_x)$ to $C_{n+1}(\Sigma_x)$ with keeping Equation (\ref{h}). It is straightforward to check that the family $\{h_{n, x} \}_{x \in X}$ is measurable. Thus, the formula $(h_n f)(u) = (h_{n, x} f_x)(u)$ ($u \in \Sigma_x^{(n)}$) defines an operator $h_n : C_n(\Sigma) \to C_{n +1}(\Sigma)$. Equation (\ref{h}) implies $\id = h_{n-1} \circ \partial_n + \partial_{n+1} \circ h_n$, that is, $\{ h_n \}_{n \geq -1}$ is a chain homotopy from $\id$ to $0$.
\end{proof}

We are ready to prove Theorem \ref{mainthm2}.

\begin{proof} (Theorem \ref{mainthm2})  
Note that $L^\infty(X)$ is $d_{L^\infty(X)}$-dense in $M(X)$, and hence the inclusion map $L^\infty(X) \hookrightarrow M(X)$ is $\dim_{L^\infty(X)}$-isomorphic so that the associated $L(G)$-map from $\Tor_n^{\mathbb{C}[G]}(L(G),L^\infty(X))$ to $\Tor_n^{\mathbb{C}[G]}(L(G),M(X))$ is also $\dim_{L(G)}$-isomorphic for every $n \geq 0$. Therefore, $\beta_n^{(2)}(G) = \dim_{L(G)} \Tor_n^{\mathbb{C}[G]} (L(G),M(X))$. With Lemma \ref{lem:alg2} and Lemma \ref{lem:auxiliary} (1), the resolution of $M(X)$ in Lemma \ref{resolution} enables us to compute
$$
\dim_{L(G)} \Tor_n^{\mathbb{C}[G]}(L(G),M(X)) = \dim_{L(G)} H_n (L(G) \otimes_{\mathbb{C}[G]} C_\bullet(\Sigma)),  
$$
which equals $\beta_n^{(2)}(\Sigma,G)$ by Proposition \ref{mainthm1}. 
\end{proof}

\begin{rem}
Berm\'udez \cite{bermudez} gave another expression of Sauer's $\beta_n^{(2)}(G)$ in terms of his generalization of the Connes-Shlyakhtenko $L^2$-Betti numbers \cite{connes-shlyakhtenko}. He defined, for an inclusion $A \subset B$ of unital $*$-algebras that is called a {\it tracial extension}, its $L^2$-Betti numbers denoted by $\beta_n^{(2)}(A / B)$. Every pmp discrete groupoid $G$ defines a tracial extension $L^\infty(X) \subset \mathbb{C} [G]$. He has proved that $\beta_n^{(2)}(\mathbb{C} [G] / L^\infty(X)) = \beta_n^{(2)}(G)$ holds for every $n \geq 0$ (\cite[Theorem 1.2]{bermudez}).
\end{rem}

\medskip
Since the universal complex $EG$ (see \S\S 3.1) is contractible, we have: 

\begin{cor} For every $n \geq 0$, we have $\beta_n^{(2)}(G) = \beta_n^{(2)}(EG,G)$. 
\end{cor} 

As in the proof of Lemma \ref{resolution}, we can also prove the following:  

\begin{cor}\label{n-connected}
If $\Sigma$ is an $n$-connected, simplicial $G$-complex, i.e., $\Sigma_x$ is $n$-connected in the usual sense (see e.g.~\cite[Chapter 1, Section 8]{spanier}) for $\mu$-a.e.~$x$, then $\beta_k^{(2)}(\Sigma, G) = \beta_k^{(2)}(G)$ as long as $0 \leq k \leq n$, and moreover, $\beta_{n +1}^{(2)}(\Sigma, G) \geq \beta_{n+1}^{(2)}(G)$.
\end{cor}

\begin{proof}
For $\mu$-a.e.~$x \in X$, the sequence 
$$C_{n+1}(\Sigma_x) \stackrel{\partial_{n+1, x}}{\to} \cdots \stackrel{\partial_{2, x}}{\to} C_1(\Sigma_x) \stackrel{\partial_{1, x}}{\to} C_0(\Sigma_x) \stackrel{\epsilon_x}{\to} \mathbb{C} \to 0$$
is exact since $\Sigma_x$ is $n$-connected. Then, by the proof of Lemma \ref{resolution}, we conclude that the sequence $C_{n+1}(\Sigma) \stackrel{\partial_{n+1}}{\to} \cdots \stackrel{\partial_{2}}{\to} C_1(\Sigma) \stackrel{\partial_{1}}{\to} C_0(\Sigma) \stackrel{\epsilon}{\to} M(X) \to 0$ is exact. Taking a projective $\mathbb{C}[G]$-resolution of $\ker \partial_{n+1}$, we get a resolution $P_\bullet$ of $M(X)$ such that $P_k$ is projective for every $k \geq n+2$. For every $k \leq n$, we have $H_k(L(G)\otimes_{\mathbb{C}[G]}P_\bullet) = H_k(L(G)\otimes_{\mathbb{C}[G]}C_\bullet(\Sigma))$, hence $\beta_k^{(2)}(G) = \beta_k^{(2)}(\Sigma,G)$. Since $\im \partial_{n+2} \subset \im (P_{n+2} \to P_{n+1})$, we get a surjective $L(G)$-map $H_{n+1}(L(G) \otimes_{\mathbb{C}[G]} C_\bullet(\Sigma)) \to H_{n+1}(L(G) \otimes_{\mathbb{C}[G]} P_\bullet)$; implying $\beta_{n+1}^{(2)}(\Sigma, G ) \geq \beta_{n+1}^{(2)}(G)$. 
\end{proof}

\section{Costs of pmp discrete groupoids} 

\subsection{Various definitions of costs and their equivalence}
We recall some definitions of costs of pmp discrete groupoids and prove their equvalence. 

\subsubsection{Measure theoretic approach}
This is a straightforward generalization of the Gaboriau's definition~\cite{gaboriau:cost} to pmp discrete groupoids. Let $\mathcal{E}$ be an at most countable family of elements of $\mathcal{G}_G$, the set of one-sheeted sets, see \S2. A non-empty element $E_1^{\epsilon_1} \cdots E_n^{\epsilon_n}$ with $E_i \in \mathcal{E}$, $\epsilon_i \in \{ 1, -1\}$ ($1 \leq i \leq n$) is called a {\it reduced word} in $\mathcal{E}$, if $E_i = E_{i+1}$ implies $\epsilon_i = \epsilon_{i +1}$ for every $1 \leq i \leq n$. Let $\mathrm{Wr}(\mathcal{E})$ denote the set of reduced words in $\mathcal{E}$. A family $\mathcal{E}$ is called a {\it graphing } of $G$ if it generates $G$ up to null set, namely 
$$\mu^G( G \setminus (X \cup \bigcup_{W \in \mathrm{Wr}(\mathcal{E})} W )) = 0$$ 
holds. The {\it cost} of a graphing $\mathcal{E}$ is defined to be 
$$C_\mu(\mathcal{E}) := \sum_{E \in \mathcal{E}} \mu^G(E) = \sum_{E \in \mathcal{E}} \mu(s(E)) = \sum_{E \in \mathcal{E}} \mu(r(E)),$$
and that of $G$ is defined to be $C_\mu(G) = \inf \{ C_\mu(\mathcal{E}) |\, \mathcal{E} : \text{graphing of }\, G \}$. 

There is another expression of costs used by Ab\'ert and Weiss \cite{abert-weiss}. A Borel subset $A \subset G$ is called a {\it generating set} of $G$ if $\mu^G( G \setminus ( \bigcup_{n \geq 1} (A \cup A^{-1} \cup X )^n ) = 0$ holds. Let $\tilde{C}_\mu(G)$ denote the number $\inf \{ \mu^G(A) \, | \, A : \text{ generating set of }\, G \}$ for temporarily. 

\begin{rem} $C_\mu(G) = \tilde{C}_\mu(G)$. 
\end{rem}
\begin{proof}
For any graphing $\mathcal{E}$ of $G$, the set $A_{\mathcal{E}} := \bigcup_{E \in \mathcal{E}} E$ is a generating set of $G$. Thus, we have $\tilde{C}_\mu(G) \leq \mu^G(A_{\mathcal{E}}) \leq \sum_{E \in \mathcal{E}} \mu^G(E) = C_\mu(\mathcal{E})$. Hence $\tilde{C}_\mu(G) \leq C_\mu(G)$. Conversely, take a generating set $A \subset G$. Let $G = \bigsqcup_{i \in I} E_i$ be a countable decomposition of $G$ into one-sheeted sets. Then $\mathcal{E}_A := \{ A \cap E_i \}_{i \in I}$ is a graphing of $G$. Thus, we have $C_\mu(G) \leq C_\mu(\mathcal{E}_A) = \sum_{i \in I} \mu^G(A \cap E_i) = \mu^G(A)$. Hence $C_\mu(G) \leq \tilde{C}_\mu(G)$.
\end{proof}
 
\subsubsection{Operator algebra approach} 
Let $(M, \tau)$ be a finite von Neumann algebra equipped with a faithful normal tracial state, $A$ be a commutative von Neumann subalgebra, and $E^M_A : M \to A$ be the $\tau$-preserving conditional expectation. The {\it normalizing groupoid} of $A$ in $M$ is defined to be the set $\mathcal{G}(M \supset A)$ of partial isometries $v \in M$ satisfying the following: (i) the support projection and the range projection belong to $A$; (ii) $v A v^\ast = A v v^\ast$. Let us recall the definition of $E^M_A$-groupoid, an operator algebraic counterpart of the set of one-sheeted sets. 

\begin{df} (\cite[Definition 2]{ueda})
An {\it $E^M_A$-groupoid} is a subset $\mathcal{G}$ of $\mathcal{G}(M \supset A)$ satisfying the following conditions: 
\begin{enumerate}
\item If $u$, $v \in \mathcal{G}$ then $u v \in \mathcal{G}$.
\item If $u \in \mathcal{G}$ then $u^{\ast} \in \mathcal{G}$.
\item Every partial isometry in $A$ belongs to $\mathcal{G}$.
\item Let $\{ u_k \}_k$ be a family of elements of $\mathcal{G}$. If both $\{u_k^\ast u_k \}_k$ and $\{u_k u_k^\ast \}_k$ are mutually orthogonal family, then $\sum_k u_k \in \mathcal{G}$ in $\sigma$-strong* topology.
\item For any $u \in \mathcal{G}$ there exists a projection $e \in A$ satisfying $e \leq u^{\ast} u$ and $E_A(u) = e u$. 
\item For any $u \in \mathcal{G}$ and $x \in M$ we have $E_A(u x u^\ast) = u E_A(x) u^\ast$.
\end{enumerate}
\end{df}
An at most countable family $\mathcal{U}$ of elements of $\mathcal{G}$ is called a {\it graphing} of $\mathcal{G}$ if $\mathcal{G}^{\prime \prime} = A \vee \mathcal{U}^{\prime \prime}$. The {\it cost } of a graphing $\mathcal{U}$ is defined to be $C_\tau (\mathcal{U}) = \sum_{u \in \mathcal{U}} \tau (u^{\ast} u)$, and that of $\mathcal{G}$ is defined to be $\inf \{ C_\tau(\mathcal{U})\, | \, \mathcal{U} : \text{a graphing of }\, \mathcal{G} \}$. 

\subsubsection{Equivalence between two approaches}
In the rest of this section, $(M, \tau)$ and $A$ are $(L(G), \tau)$ and $L^\infty(X)$, respectively. Define $\mathcal{G}(G)$ to be the set of elements $u \in M$ of the form $u = a u(E)$ where $a$ is a partial isometry in $A$ and $E$ is a one-sheeted set of $G$. It is easy to see that $\mathcal{G}(G)$ is an $E^M_A$-groupoid and that  $\mathcal{G}(G)^{\prime \prime} = M$. The next lemma, which is missing in \cite{ueda}, guarantees the equivalence between above two approaches. 

\begin{lem}\label{equivalence}
$C_{\tau_G}(\mathcal{G}(G)) = C_\mu(G)$.
\end{lem}

\begin{proof}
Let $\mathcal{U}$ be a graphing of $\mathcal{G}(G)$. Then, for each $u \in \mathcal{U}$, there exist a partial isometry $a_u \in A$ and $E_u \in \mathcal{G}_G$ such that $u = a_u u(E_u)$. We show that $\mathcal{E}_{\mathcal{U}} := \{ E_u \}_{u \in \mathcal{U}}$ is a graphing of $G$. Suppose that this is not the case, that is, $\mu^G(G \setminus (X \cup \bigcup_{W \in \mathrm{Wr}(\mathcal{E}_{\mathcal{U}})} W )) > 0$. Then, there exists a  non-null one-sheeted set $F$ of $G$ such that $F \subset G \setminus (X \cup \bigcup_{W \in \mathrm{Wr}(\mathcal{E}_{\mathcal{U}})} W )$. Since $\mu(s(F)) = \mu^G(F) \neq 0$, we have $u(F)^* u(F) = \mathbbm{1}_{s(F)} \neq 0$ and hence $u(F) \neq 0$. On the other hand, since $F \cap (X \cup \bigcup_{W \in \mathrm{Wr}(\mathcal{E}_{\mathcal{U}})} W) = \emptyset$, we have $E^M_A(u(F)) = \mathbbm{1}_{X \cap F} = 0$ and $E^M_A(u(W)^* u(F)) = E^M_A(u(W^{-1} \cdot F)) = \mathbbm{1}_{s(W \cap F)} = 0$ for every $W \in \mathrm{Wr}(\mathcal{E}_{\mathcal{U}})$. Thus, by \cite[Lemma 3]{ueda}, we have $u(F) = 0$, which contradicts $u(F) \neq 0$. Hence $\mathcal{E}_{\mathcal{U}}$ is a graphing of $G$. Then, one computes $C_{\tau_G}(\mathcal{U}) = \sum_{u \in \mathcal{U}} \tau_G( u(E_u)^* a_u^* a_u u(E_u)) = \sum_{u \in \mathcal{U}} \tau_G( u(E_u^{-1} E_u)) = \sum_{u \in \mathcal{U}} \mu(s(E_u)) = C_\mu(\mathcal{E}_{\mathcal{U}}) \geq C_\mu(G)$. Since this inequality holds for every graphing $\mathcal{U}$ of $\mathcal{G}(G)$, we obtain $C_{\tau_G}(\mathcal{G}(G)) \geq C_\mu(G)$. 

Let $\mathcal{E}$ be a graphing of $G$. We show that $\mathcal{U}_{\mathcal{E}} := \{ u(E) \, | \, E \in \mathcal{E} \}$ is a graphing of $\mathcal{G}(G)$. Let $ \{ W_j \}_{j \geq 0}$ be an enumeration of $\mathrm{Wr}(\mathcal{E}) \cup \{ X \}$ with $W_0 = X$. Define a family $\{ \tilde{W_j} \}_{j \geq 0}$ inductively by $\tilde{W_0} = W_0$ and $\tilde{W_n} = W_n \setminus (\bigcup_{j=0}^{n-1} \tilde{W_j})$. Then $G = \bigsqcup_{j \geq 0} \tilde{W_j}$ up to null set. Take $E \in \mathcal{G}_G$. Since $E = \bigsqcup_{j \geq 0} (E \cap \tilde{W_j})$ up to null set,  we have $u(E) = \sum_{j \geq 0} u(E \cap \tilde{W_j})$ in the $\sigma$-strong operator topology. Since $E \cap \tilde{W_j} \subset W_j$, we have $u(E \cap \tilde{W_j}) = \mathbbm{1}_{r(E \cap \tilde{W_j})} u (W_j) \in A \vee \mathcal{U}_{\mathcal{E}}^{\prime \prime} $ for every $j \geq 0$. Thus $u(E) \in A \vee \mathcal{U}_{\mathcal{E}}^{\prime \prime}$ for every $E \in \mathcal{G}_G$. Hence, we conclude that $M = A \vee \mathcal{U}_{\mathcal{E}}^{\prime \prime}$, that is, $\mathcal{U}_{\mathcal{E}}$ is a graphing of $\mathcal{G}(G)$. Then, one compute $C_\mu(\mathcal{E}) = \sum_{E \in \mathcal{E}} \mu(s(E)) = \sum_{E \in \mathcal{E}} \tau_G( u(E)^* u(E)) = C_{\tau_G}(\mathcal{U}_{\mathcal{E}}) \geq C_{\tau_G} ( \mathcal{G}(G))$. Since the inequality holds for every graphing  $\mathcal{E}$ of $G$, we obtain $C_\mu(G) \geq C_{\tau_G}(\mathcal{G}(G))$.  
\end{proof}

\subsection{Some properties of groupoid cost}
We prove that three important results of Gaboriau \cite{gaboriau:cost} hold true even for arbitrary pmp discrete groupoids. The first two (Proposition \ref{induction} and Theorem \ref{additivity}) are proved by translating the corresponding results in \cite{ueda} into pmp discrete groupoid setting, though one can prove them in the framework of groupoids directly by translating the proofs in \cite{ueda} into the framework. The last one (Theorem \ref{treeing_attains_cost}), which is a central result in the theory of costs, is proved directly because it is missing in \cite{ueda}.

\subsubsection{ Induction formula.} For any Borel subsets $Y_1$, $Y_2 \subset X$, the symbol $G^{Y_1}_{Y_2}$ denotes the set $s^{-1}(Y_1) \cap r^{-1}(Y_2)$. The {\it  restriction} $G \upharpoonright_{Y}$ of $G$ to a Borel subset $Y \subset X$ is defined to be $G^Y_Y$.

An at most countable family $\mathcal{E} \subset \mathcal{G}_G$ is called a {\it treeing} of $G$ if $\mu^G(W \cap X) = 0$ for every reduced word $W$ in $\mathcal{E}$. For an $E^M_A$-groupoid $\mathcal{G}$, an at most countable family $\mathcal{U} \subset \mathcal{G}$ is called a {\it treeing} if $E^M_A (w) = 0$ for every reduced word $w$ in $\mathcal{U}$. A pmp discrete groupoid $G$ (resp. an $E^M_A$-groupoid $\mathcal{G}$) is said to be {\it treeable} if it has a treeing which is also a graphing.  Note that $E^M_A(u(E)) = 0$ if and only if $\mu^G(E \cap X) = 0$. Indeed, it is easy to see that $E^M_A(u(E)) = u(E \cap X)$. 

We prove the following proposition: 

\begin{prop}\label{induction}(groupoid version of \cite[Proposition II. 6]{gaboriau:cost})
Let $Y \subset X$ be a Borel subset satisfying $\# G^x_Y \geq 1$ for $\mu$-a.e.~$x \in X$. Then, we have the following: 
\begin{enumerate}
\item $C_\mu(G) - 1 = C_\mu(G \upharpoonright_Y) - \mu(Y)$;
\item $G$ is treeable if and only if so is $G \upharpoonright_Y$. 
\end{enumerate}
\end{prop}

The next lemma seems standard, but we do give its proof for the sake of completeness.

\begin{lem}\label{induction_lem}
For a Bore subset $Y \subset X$, the inequality $\# G^x_Y \geq 1$ holds for $\mu$-a.e.~$x \in X$ if and only if the central support projection $z_M( \mathbbm{1}_Y) = 1$.
\end{lem}
\begin{proof}	
Suppose that $z_M(\mathbbm{1}_Y) = 1$. The set $\tilde{Y} := \{ x \in X \, | \, \# G^x_Y \geq 1 \}$ is a $G$-invariant Borel subset that contains $Y$. Thus, we have $\mathbbm{1}_Y \leq \mathbbm{1}_{\tilde{Y}}$, which is a central projection in $M$. Hence $1 = z_M(\mathbbm{1}_Y) \leq \mathbbm{1}_{\tilde{Y}}$, that is, $\mathbbm{1}_{\tilde{Y}} = 1$. This implies that $\# G^x_Y \geq 1$ holds for $\mu$-a.e.~$x \in X$. 

Conversely, suppose that $\# G^x_Y \geq 1$ holds for $\mu$-a.e.~$x \in X$. Then, $G \cdot Y := r(s^{-1}(Y))$ is a conull subset, thus $\mathbbm{1}_{G \cdot Y} = 1$. Let $G = \bigsqcup_{i \geq 1} E_i$ be a decomposition into one-sheeted sets; see \S 2. Then, we have $G \cdot Y = \bigcup_{i \geq 1} \varphi_{E_i} (Y)$. Thus, we have $1 = \mathbbm{1}_{G \cdot Y} = \bigvee_{i \geq 1} \mathbbm{1}_{\varphi_{E_i}(Y)} = \bigvee_{i \geq 1} u(E_i) \mathbbm{1}_Y u(E_i)^*$. On the other hand, by an explicit description of the central support, we have $u(E_i) \mathbbm{1}_Y u(E_i)^* \leq z_M(\mathbbm{1}_Y)$ for every $i \geq 1$. Therefore we have $z_M(\mathbbm{1}_Y) = 1$.
\end{proof}

\begin{proof} (Proposition \ref{induction})
(1) We have $z_M(\mathbbm{1}_Y) = 1$ by Lemma \ref{induction_lem}. Applying \cite[Proposition 15]{ueda}, we get $C_\tau(\mathcal{G}(G)) - 1 = C_{\tau \upharpoonright_{\mathbbm{1}_Y M \mathbbm{1}_Y}} ( \mathbbm{1}_Y \mathcal{G}(G) \mathbbm{1}_Y) - \tau (\mathbbm{1}_Y)$. It is not hard to see that  $\mathbbm{1}_Y \mathcal{G}(G) \mathbbm{1}_Y = \mathcal{G}(G \upharpoonright_Y)$ and that $\mathbbm{1}_Y M \mathbbm{1}_Y = L(G \upharpoonright_Y)$. Thus, applying Lemma \ref{equivalence}, we get an equality $C_\mu(G) - 1 = C_\mu(G \upharpoonright_Y) - \mu(Y)$.

(2) Thanks to \cite[Proposition 15]{ueda}, it suffices to show that $G$ is treeable if and only if so is $\mathcal{G}(G)$. The only if part is easy. Let $\mathcal{U}$ be a treeing of $\mathcal{G}(G)$ and $\mathcal{E}_{\mathcal{U}}$ be its associated graphing of $G$ (see the proof of \ref{equivalence}). Then, the family $\{ A \vee \{ u \}^{\prime \prime} \}_{u \in \mathcal{U}}$ is a free family of von Neumann algebra with respect to $E^M_A$; see \cite[\S 3.8]{voiculescu-dykema-nica} for the definition of freeness. Since $u(E_u) \in A \vee \{ u \}^{\prime \prime}$ for every $u \in \mathcal{U}$, the freeness of $\{ A \vee \{ u \}^{\prime \prime} \}_{u \in \mathcal{U}}$ implies that $\mathcal{E}_{\mathcal{U}}$ is a treeing of $G$. Hence we are done.
\end{proof}

\subsubsection{ Additivity formula.} Let $G_1 \supset G_3 \subset G_2$ be subgroupoids of a pmp discrete groupoid $G$ with $G_3 = G_1 \cap G_2$. We say that $G$ is the free product of $G_1$ and $G_2$ with amalgamation $G_3$ and write $G = G_1 \bigstar_{G_3} G_2$ if the following conditions are satisfied: $G$ is generated by $G_1$ and $G_2$; for any alternating word $E_1 \cdots E_n$ in $\mathcal{G}(G_1)$ and $\mathcal{G}(G_2)$ satisfying $\mu^G( E_i \cap G_3) = 0$ for every $i \geq 1$, we have $\mu^G( (E_1 \cdots E_n) \cap G_3) = 0$. A rigorous (i.e., measurable) construction of free products with amalgamations was given in \cite{kosaki}, but we do not need it here. 

\begin{thm}\label{additivity}(groupoid version of \cite[Th\'eor\`eme IV. 15]{gaboriau:cost})
Let $G_1 \supset G_3 \subset G_2$ be subgroupoids of a discrete pmp groupoid $G$ with $G_3 = G_1 \cap G_2$. Assume that $G = G_1 \bigstar_{G_3} G_2$ and that $G_3$ is principal and  hyper finite. Assume further that both $C_\mu(G_1)$ and $C_\mu(G_2)$ are finite.Then, $C_\mu(G) = C_\mu(G_1) + C_\mu(G_2) - C_\mu(G_3)$ holds.  
\end{thm}

\begin{proof}
We use the following notation: $\mathcal{G}_i = \mathcal{G}(G_i)$, $N_i = \mathcal{G}_i^{\prime \prime} = L(G_i)$. In order to apply \cite[Theorem 9]{ueda} to our situation, we show the following assertions: 
\begin{enumerate}
\item $(M, E^M_{N_3}) = (N_1, E^M_{N_3} \upharpoonright_{N1}) \bigstar_{N_3} (N_2, E^M_{N_3} \upharpoonright_{N2})$; 
\item $N_3$ is a hyperfinite von Neumann algebra that contains $A$ as a MASA;
\item the smallest $E^M_A$-groupoid $\mathcal{G}_1 \vee \mathcal{G}_2$ which contains $\mathcal{G}_1$ and $\mathcal{G}_2$ equals $\mathcal{G}(G)$.
\end{enumerate}
 
(1) First, we show that $M$ is generated by $N_1$ and $N_2$. Let $\mathcal{E}_i$ be a graphing of each $G_i$. Since $G$ is generated by $G_1$ and $G_2$, we have $\mu^G( G \setminus (X \cup \bigcup_{W \in \mathrm{Wr}( \mathcal{E}_1 \cup \mathcal{E}_2) } W)) = 0$. Then, by an argument similar to that in the proof of Lemma \ref{equivalence}, we conclude that $u(\mathcal{G}_G) \subset N_1 \vee N_2$. Thus $M = N_1 \vee N_2$. 

Next, we show that $u(\mathcal{G}_{G_1})$ and $u(\mathcal{G}_{G_2})$ are $*$-free with amalgamation $N_3$ with respect to $E^M_{N_3}$. It is not hard to see that $E^M_{N_3}(u(E)) = u(E \cap G_3)$ for every $E \in \mathcal{G}_G$. Thus, $\mu^G(E \cap G_3) = 0$ if and only if $E^M_{N_3}(u(E)) = 0$; this fact enables us to show the assertion. 

(2) Since $G_3$ is principal, $G_3$ is nothing but a pmp discrete equivalence relation. Hence, $N_3$ is a hyperfinite von Neumann algebra that contains $A$ as a MASA; see \cite[Proposition 2.9]{feldman-moore}. 

(3) Let $\mathcal{E}_i$ be a graphing of each $G_i$. Then, by the proof of Lemma \ref{equivalence}, $\mathcal{U}_i := u(\mathcal{E}_i)$ is a graphing of $\mathcal{G}_i$. Also, we have proved that $M = N_1 \vee N_2$. Thus, $\mathcal{U} := \mathcal{U}_1 \cup \mathcal{U}_2$ is a graphing of $\mathcal{G}(G)$. Therefore, for every $u \in \mathcal{G}(G)$, by \cite[Lemma 3]{ueda}, there exists a family $\{ u_w \}_{w \in \mathrm{Wr}(\mathcal{U})} \subset \mathcal{G}(G)$ satisfying the following: (i) every $u_w$ is a product of a partial isometry in $A$ and a reduced word in $\mathcal{U}$; (ii) the support projections and range projections respectively form mutually orthogonal families; (iii) $u = \sum_{w \in \mathrm{Wr}(\mathcal{U})} u_w$ in the $\sigma$-strong$^*$ topology. Since each $u_w$ belongs to $\mathcal{G}_1 \vee \mathcal{G}_2$, the above condition (ii) implies that $u \in \mathcal{G}_1 \vee \mathcal{G}_2$. 

Hence we can apply \cite[Theorem 9]{ueda} to our $E^M_A$-groupoids $\mathcal{G}_1 \supset \mathcal{G}_3 \subset \mathcal{G}_2$. Then, by Lemma \ref{equivalence}, we conclude that $C_\mu(G) = C_\mu(G_1) + C_\mu(G_2) - C_\mu(G_3)$ holds if both $C_\mu(G_2)$ and $C_\mu(G_3)$ are finite.
\end{proof}

\subsubsection{Any treeing attains the cost} 

\begin{thm} \label{treeing_attains_cost} (groupoid version of \cite[Th\'eor\`eme IV. 1]{gaboriau:cost}) 
If $G$ is generated by a treeing $\mathcal{E}$, then we have $C_\mu(G) = C_\mu(\mathcal{E})$.
\end{thm}

To prove the theorem, we provide a terminology and lemmas. A Borel subset $A \subset X$ is said to be G-{\it invariant} if $r(s^{-1}(A)) \subset A$.  

\begin{lem}\label{cost_decomp}
If $X = \bigsqcup_{i \in I} X_i$ is a countable Borel partition by $G$-invariant sets, then we have $C_\mu(G) = \sum_{i \in I} C_\mu(G \upharpoonright_{X_i})$. 
\end{lem} 
\begin{proof}
Let $\mathcal{E}$ be a graphing of $G$. Since each $X_i$ is $G$-invariant, the family $\mathcal{E}_i := \{ s^{-1}(s(E) \cap X_i) \, | \, E \in \mathcal{E} \}$ is a graphing of each $G \upharpoonright_{X_i}$. Then $C_\mu(\mathcal{E}) = \sum_{i \in I} C_\mu(\mathcal{E}_i) \geq \sum_{i \in I} C_\mu(G \upharpoonright_{X_i})$. Thus $C_\mu(G) \geq \sum_{i \in I} C_\mu(G \upharpoonright_{X_i})$. Conversely, let $\mathcal{E}_i$ be a graphing of each $G \upharpoonright_{X_i}$. Then, $\bigcup_{i \in I} \mathcal{E}_i$ is a graphing of $G$, and hence $C_\mu(G) \leq \sum_{i \in I} C_\mu(\mathcal{E}_i)$. Hence we have $C_\mu(G) \leq \sum_{i \in I} C_\mu(G \upharpoonright_{X_i})$.
\end{proof}

Let $\Gamma$ be a discrete group, $X \times \Gamma \ni (x, \gamma) \mapsto x \gamma \in X$ be a (not necessarily, essentially free) pmp action on a probability space $(X, \mu)$. Define a discrete groupoid $X \rtimes \Gamma$ as follows: $X \rtimes \Gamma = X \times \Gamma$ as a Borel space, where $\Gamma$ is endowed with the discrete Borel structure, and the groupoid operations are defined in the following manner: $s : (x, \gamma) \mapsto x \gamma$, $r : (x, \gamma) \mapsto x$ and $(x, \gamma_1) (x \gamma_1, \gamma_2) := (x, \gamma_1 \gamma_2)$. This discrete groupoid clearly becomes pmp with $\mu$. We call the groupoid $X \rtimes \Gamma$ the {\it transformation groupoid} associated with the action. 

\begin{lem}\label{trivial_action}
For any finite measure space $(Y, \nu)$, we have $C_\nu(Y \rtimes_{\id} \mathbb{Z}) = \nu(Y)$. 
\end{lem}
\begin{proof}
Since $\{ Y \times \{ 1 \} \}$ is a graphing of $Y \rtimes_{\id} \mathbb{Z}$, we have $C_\nu(Y \rtimes_{\id} \mathbb{Z}) \leq \nu(Y)$. Conversely, take an arbitrary graphing $\mathcal{E}$ of $Y \rtimes_{id} \mathbb{Z}$. Since the action $\mathbb{Z} \curvearrowright Y$ is trivial, we have $\nu(Y \setminus \bigcup_{E \in \mathcal{E}} s(E)) = 0$. Thus, we have $\nu(Y) \leq \sum_{E \in \mathcal{E}} \nu(s(E)) = C_\nu(\mathcal{E})$. Hence $C_\nu(Y \rtimes_{\id} \mathbb{Z}) \geq \nu(Y)$.
\end{proof}

Let $\mathcal{R}_G$ denote the pmp discrete equivalence relation defined to be $(r \times s) (G)$. 

\begin{lem}\label{groupoid_geq_equiv} We have $C_\mu(G) \geq C_\mu(\mathcal{R}_G)$.
\end{lem}
\begin{proof}
Take an arbitrary graphing $\mathcal{E}$ of $G$. Then, $\Phi_{\mathcal{E}} := \{ \varphi_E \}_{E \in \mathcal{E}}$ is a graphing of $\mathcal{R}_G$. We have $C_\mu(\mathcal{E}) = C_\mu(\Phi_{\mathcal{E}}) \geq C_\mu(\mathcal{R}_G)$. Hence we have $C_\mu(G) \geq C_\mu(\mathcal{R}_G)$.    
\end{proof}

The next lemma is a special case of Theorem \ref{treeing_attains_cost}. 

\begin{lem}\label{one_element}
If $G$ is generated by a single treeing $\{ E \}$ which consists of one element, then we have $C_\mu(G) = C_\mu( \{ E \}) = \mu(s(E))$.
\end{lem}
\begin{proof}
Since $\{ E \}$ is a graphing of $G$, we have $C_\mu(G) \leq \mu(s(E))$.

We show the converse inequality. Let $\mathcal{R}_G$ be the pmp discrete equivalence relation associated with $G$, that is, $(x, y) \in \mathcal{R}_G$ if and only if $y = \varphi_E^n(x)$ for some $n \in \mathbb{Z}$. Set $Y := s(E) \cup r(E)$ and $X_0 := X \setminus Y$. Define $X_n := \{ x \in Y\, | \, \# \mathcal{R}_G (x) = n \}$ for every $1 \leq n \leq \infty$. The family $\{ X_n \}_{0 \leq n \leq \infty}$ gives a $G$-invariant partition of $X$, thus Lemma \ref{cost_decomp} implies 
$$C_\mu(G) = C_\mu(G \upharpoonright_{X_0}) + \sum_{n \geq 1} C_\mu(G \upharpoonright_{X_n}) + C_\mu(G \upharpoonright_{X_\infty}).$$ 
We compute each term below.

\medskip
({\bf First term}: $C_\mu(G \upharpoonright_{X_0}) = 0$.) This is trivial since $G \upharpoonright_{X_0} = X_0$.

\medskip
({\bf Second term}: $C_\mu(G \upharpoonright_{X_n}) = \mu(X_n \cap r(E))$ for every $1 \leq n < \infty$.) Define a Borel subset $D_n \subset  s(E)$ for every $1 \leq n \leq \infty$ as follows: $D_n := \operatorname{Dom} (\varphi_E^n) \setminus \operatorname{Dom}(\varphi_E^{n + 1})$ for $1 \leq n < \infty$ and $D_\infty := \bigcap_{n \geq 1} \operatorname{Dom}(\varphi_E^n)$; we have $D = \bigsqcup_{n \geq 1} D_n \sqcup D_\infty$. Since $X_n = (X_n \cap D_\infty) \bigsqcup (X_n \setminus (X_n \cap D_\infty))$ is a $G$-invariant partition, we have $C_\mu(G \upharpoonright_{X_n}) = C_\mu(G \upharpoonright_{X_n \cap D_\infty}) + C_\mu(G \upharpoonright_{X_n \setminus (X_n \cap D_\infty)})$. 

First, we compute the first term. Let $F_n \subset X_n \cap D_\infty$ be a fundamental domain for $\mathcal{R}_G \upharpoonright_{X_n \cap D_\infty}$. Then, by the induction formula (Proposition \ref{induction}), we have $C_\mu(G \upharpoonright_{X_n \cap D_\infty}) - \mu(X_n \cap D_\infty) = C_\mu(G \upharpoonright{F_n}) - \mu(F_n)$. Since $\{ E \}$ is a treeing, we have $G = \bigsqcup_{k \in \mathbb{Z}} E^k$, a disjoint union, with $E^0 = X$, and then $G \upharpoonright_{F_n} = \bigsqcup_{k \in \mathbb{Z}} G \upharpoonright_{F_n} \cap E^{n k}$. For every $k \in \mathbb{Z}$, define a homomorphism $G \upharpoonright_{F_n} \cap E^{n k} \to F_n \rtimes_{\id} \mathbb{Z} : g \mapsto (s(g), k)$, giving an isomorphism $G \upharpoonright_{F_n} \to F_n \rtimes_{\id} \mathbb{Z}$.  thus Lemma \ref{trivial_action} implies that $C_\mu(G \upharpoonright_{X_n \cap D_\infty}) = \mu(X_n \cap D_\infty)$. 

Next, we compute the second term. Note that $X_n \setminus (X_n \cap D_\infty) = \bigsqcup_{k = 1}^{n -1}(X_n \cap D_k) \sqcup \varphi_E(X_n \cap D_1)$ and that $X_n \cap D_{n -1}$ is a fundamental domain for $\mathcal{R}_G \upharpoonright_{X_n \setminus (X_n \cap D_\infty)}$. Since $G \upharpoonright_{X_n \cap D_{n-1}} = X_n \cap D_{n-1}$, the induction formula implies that $C_\mu(G \upharpoonright_{X_n \setminus (X_n \cap D_\infty)}) = \mu(X_n \setminus (X_n \cap D_\infty)) - \mu(X_n \cap D_{n-1})$. 

Hence we have $C_\mu(G \upharpoonright_{X_n}) = \mu(X_n \setminus (X_n \cap D_{n-1}))$. The definition of $\{X_n \}_{1 \leq n < \infty}$ and $\{ D_n \}_{ 1 \leq n \leq \infty}$ implies that $X_n \setminus (X_n \cap D_{n-1}) = X_n \cap r(E)$. Thus we have $C_\mu(G \upharpoonright_{X_n}) = \mu(X_n \cap r(E))$. 

\medskip
({\bf Third term}: $C_\mu(G \upharpoonright_{X_\infty}) \geq \mu(X_\infty \cap r(E))$. ) The definition of $X_\infty$ implies that $\mathcal{R}_G$ is an aperiodic (i.e., every orbit is an infinite set) equivalence relation. Thus, by \cite[Proposition III.3 (1)]{gaboriau:cost} and Lemma \ref{groupoid_geq_equiv}, we conclude that $C_\mu(G \upharpoonright_{X_\infty}) \geq \mu(X_\infty \cap r(E))$.

\medskip
Therefore, we have the inequality  
$C_\mu(G) \geq \sum_{n \geq 1} \mu(X_n \cap r(E)) + \mu(X_\infty \cap r(E)) = \mu(Y \cap r(E)) = \mu(r(E)) = C_\mu(\{ E \})$, which completes the proof.
\end{proof}

We are ready to prove Theorem \ref{treeing_attains_cost}. 

\begin{proof}(Theorem \ref{treeing_attains_cost})
 Let $\mathcal{E} = \{ E_i \}_{i =1}^N$ be a treeing which generates $G$. For every $1 \leq i \leq N$, the symbol $G_{E_i}$ denotes the groupoid generated by $E_i$.

First, consider the case when $N$ is finite. Since $\mathcal{E}$ is a treeing, the groupoid $G$ is the free product $G_{E_1} \bigstar_X \cdots \bigstar_X G_{E_N}$. For every $1 \leq i \leq N$ we have $C_\mu(G_{E_i}) = \mu(s(E_i)) < \infty$ by Lemma \ref{one_element}. Thus, by the additivity formula (Theorem \ref{additivity}), we have $C_\mu(G) = \sum_{i=1}^N \mu(s(E_i)) = C_\mu(\mathcal{E})$. 

Next, consider the case when $N = \infty$. Take an arbitrary graphing $\mathcal{F} = \{ F_i \}_{i \geq 1}$ of $G$. We show $C_\mu(\mathcal{F}) \geq C_\mu(\mathcal{E})$. As in the proof of \cite[IV.39. Th\'eor\`em IV.1]{gaboriau:cost}, (decomposing each one-sheeted set if necessary) we may and do assume that every $F_i$ is a subset of a reduced word in $\mathcal{E}$. Fix $n \geq 1$. Since $\mathcal{F}$ is a graphing of $G$, there exists an integer $k(n) \geq 1$ and Borel subsets $\tilde{E}_1 \subset E_1, \dots, \tilde{E}_n \subset E_n$ satisfying the following conditions for every $1 \leq j \leq n$: $\mu^G(\tilde{E}_j) \leq 2^{-n}$; any element of $E_j \setminus \tilde{E}_j$ belongs to some word in $\mathcal{F}_{k(n)}$. On the other hand, there exists $m \geq n$ so that every element of $\mathcal{F}_{k(n)}$ is a subset of a word in $\mathcal{E}_m$. Thus, the family $\tilde{\mathcal{F}} := \mathcal{F}_{k(n)} \sqcup \{ \tilde{E}_j \}_{j = 1}^n \sqcup (\mathcal{E}_m \setminus \mathcal{E}_n)$ is a graphing of $G_{\mathcal{E}_m}$. We have $C_\mu(\mathcal{F}_{k(n)}) + \sum_{j=1}^n \mu^G(\tilde{E}_j) + C_\mu(\mathcal{E}_m \setminus \mathcal{E}_n) = C_\mu(\tilde{\mathcal{F}}) \geq C_\mu(G_{\mathcal{E}_m}) = C_\mu(\mathcal{E}_m) \geq C_\mu(\mathcal{E}_n)$. Here the last equality follows from what we have proved in the previous paragraph. Hence the inequality $C_\mu(\mathcal{F}) \geq C_\mu(\mathcal{E}_n) - n 2^{-n}$ holds for every $n \geq 1$, thus we conclude that $C_\mu(\mathcal{F}) \geq C_\mu(\mathcal{E})$. Therefore, we have $C_\mu(\mathcal{G}) = C_\mu(\mathcal{E})$. 
\end{proof}

The converse of Theorem \ref{treeing_attains_cost} is not true; in \cite[Remark 12 (1)]{ueda} it was pointed out (with a simple example) that \cite[Proposition I.11]{gaboriau:cost}, a result asserting ``any graphing attaining the cost is a treeing'', does not hold in the groupoid setting. 

\begin{cor}\label{free_group_action}
Let $n \in \mathbb{N} \cup \{ \infty \}$. For any (not necessarily essentially free) pmp action of the free group $\mathbb{F}_n$ on a probability space $(X, \mu)$, we have $C_\mu(X \rtimes \mathbb{F}_n) = n$. 
\end{cor}
\begin{proof}
Let $\{ a_i \}_{i=1}^n$ be a free generator of $\mathbb{F}_n$. Then $\mathcal{E} := \{ X \times \{a _i\} \}_{i=1}^n$ is a treeing of $X \rtimes \mathbb{F}_n$, thus we have $C_\mu(X \rtimes \mathbb{F}_n) = C_\mu(\mathcal{E}) = n$. 
\end{proof}

We will give another explanation of the Corollary in \S\S\S 5.3.1.


\section{The Morse inequalities and its corollaries} 

\subsection{The Morse inequalities}
Let $\Sigma$ be a simplicial $G$-complex. For simplicity, define $\alpha_k(\Sigma) := \dim_{L(G)} C_k^{(2)}(\Sigma)$. We prove the following theorem: 

\begin{thm} \label{morse} (groupoid version of \cite[Proposition 3.19]{gaboriau:betti}) 
Assume that the number $\alpha_k(\Sigma)$ is finite for every $0 \leq k \leq n$. Then, we have 
\begin{align*}
\alpha_n(\Sigma) - &\alpha_{n -1}(\Sigma) + \dots + (-1)^n \alpha_0(\Sigma) \leq
\beta_n^{(2)} ( \Sigma, G) - \beta_{n-1}^{(2)}(\Sigma, G) + \dots + (-1)^n \beta_0^{(2)}(\Sigma, G).
\end{align*}
\end{thm}

To prove the theorem, we need the following general lemma. Let $(M, \tau)$ be a finite von Neumann algebra equipped with a faithful normal tracial state. A morphism $f : V \to W$ between Hilbert $M$-modules is called an {\it $\epsilon$-isomorphism} for $\epsilon > 0$, if both $\dim_M \ker f$ and $\dim_M W / \overline{\im_M f}$ are not larger than $\epsilon$. The next lemma is shown in the exactly same way as in \cite[Lemme 4.2]{gaboriau:betti}.

\begin{lem}\label{hilb.module}
Let $V_\bullet$ and $W_\bullet$ are Hilbert chain $M$-complexes such that both $\dim_M V_n$ and $\dim_M W_n$ are finite for every $n \geq 0$. Assume that there exists a chain morphism $\iota_\bullet : V_\bullet \to W_\bullet$ that consists of inclusions and that 
$$d(V_\bullet, W_\bullet) := \sum_{n = 0}^\infty | \dim_M V_n - \dim_M W_n |$$
is finite. Then, the induced morphism $H_n^{(2)}(\iota_\bullet) : H_n^{(2)}(V_\bullet) \to H_n^{(2)}(W_\bullet)$ is a $d(V_\bullet, W_\bullet)$-isomorphism for every $n \geq 0$.
\end{lem}  

If simplicial $G$-complexes $\Sigma \subset \Sigma^\prime$ are ULB, then we can apply the above lemma to Hilbert chain $L(G)$-complexes $C_\bullet^{(2)}(\Sigma)$ and $C_\bullet^{(2)}(\Sigma^\prime)$. Indeed, we have already seen that $\alpha_k(\Sigma)$ and $\alpha_k(\Sigma^\prime)$ are finite for every $k \geq 0$; see the paragraph posterior to Lemma \ref{lem:auxiliary}. Also, the number $d(C_\bullet^{(2)}(\Sigma), C_\bullet^{(2)}(\Sigma^\prime))$ is finite since every ULB simplicial $G$-complex is finite dimensional. 

\begin{proof}(Theorem \ref{morse})
First, we consider the case when $\Sigma$ is ULB. Applying \cite[Lemma 1.18]{luck:survey} to a Hilbert chain $L(G)$-complex $C_\bullet^{(2)}(\Sigma)$, we have $\alpha_k(\Sigma) = \beta_k^{(2)}(\Sigma, G) + b_{k+1} + \dim_{L(G)} \overline{\im {\partial_k^{(2)}}^*}$ where $b_k$ denotes $\dim_{L(G)} \overline{\im \partial_k^{(2)}}$. Since the $L(G)$-map ${\partial_k^{(2)}}^* : \overline{\im \partial_k^{(2)}} \to \overline{\im {\partial_k^{(2)}}^*}$ is an injection with dense range, the last term equals $b_k$. Hence we have $\alpha_n(\Sigma) - \alpha_{n -1}(\Sigma) + \dots + (-1)^n \alpha_0(\Sigma) - ( \beta_n^{(2)} ( \Sigma, G) - \beta_{n-1}^{(2)}(\Sigma, G) + \dots + (-1)^n \beta_0^{(2)}(\Sigma, G) ) = b_{n+1} \geq 0$. 

Next, consider the general case. Take a ULB-exhaustion $\{ \Sigma_i \}_{i \geq 1}$ of $\Sigma$. Then, the family $\{ C_k^{(2)}(\Sigma_i) \}_{i \geq 1}$ is an increasing sequence of closed subspaces of $C_k^{(2)}(\Sigma)$ with dense union. Therefore, by \cite[Theorem 1.12, (3)]{luck:survey}, we have $\alpha_k(\Sigma) = \lim_{i \to \infty} \alpha_k(\Sigma_i)$ for every $k \geq 0$. Then, we can also show that $\lim_{i \to j} d(C_\bullet^{(2)}(\Sigma_i), C_\bullet^{(2)}(\Sigma_j)) = 0$ for every $j \geq 1$. The additivity of $\dim_{L(G)}$ and Lemma \ref{hilb.module} imply that $| \beta_k^{(2)}(\Sigma_i, G) - \nabla_k(\Sigma_i, \Sigma_j) | \leq d(C_\bullet^{(2)}(\Sigma_i), C_\bullet^{(2)}(\Sigma_j))$ for every $k \geq 0$ and $j \geq i$. Hence we have $\beta_k^{(2)}(\Sigma, G) = \lim_{i \to \infty} \beta_k^{(2)}(\Sigma_i, G)$ for every $k \geq 0$. Combining the ULB case and two equalities which we have proved in the paragraph, we get the inequality for $\Sigma$.  
\end{proof}

We define $\chi(\Sigma) := \sum_{n\geq 0} ( -1 )^n \alpha_n(\Sigma)$ as long as it is well-defined, that is, it converges. Similarly, we define $\chi^{(2)}(\Sigma) = \sum_{n \geq 0} (-1)^n \beta_n^{(2)}(\Sigma, G)$ as long as it is well-defined. 

\begin{cor}\label{euler-poincare} (groupoid version of \cite[Proposition 3.20]{gaboriau:betti})
If $\chi(\Sigma)$ is well-defined, then so is $\chi^{(2)}(\Sigma)$ and these two quantities must coincide.
\end{cor}

\begin{proof}
Theorem \ref{morse} shows that
$$
\sum_{k = 0}^{2 n + 1} (-1)^k \alpha_k(\Sigma) \leq \sum_{k=0}^{2 n + 1} (-1)^k \beta_k^{(2)}(\Sigma, G) \leq \sum_{k=0}^{2 n} (-1)^k \beta_k^{(2)}(\Sigma, G) \leq \sum_{k=0}^{2 n} (-1)^k \alpha_k(\Sigma); 
$$
implying the desired result. 
\end{proof}

\subsection{Cost versus $L^2$-betti numbers inequality}
We prove the following theorem: 
\begin{thm} \label{costvsbetti}(groupoid version of \cite[Corollaire 3.23]{gaboriau:betti})
We have $\beta_1^{(2)}(G) - \beta_0^{(2)}(G) + 1 \leq C_\mu(G)$. Equality holds if $G$ is treeable.
\end{thm}

To prove the theorem, we provide a terminology and lemmas. We say that a graphing of $G$ is {\it disjoint} if it is a disjoint family.

\begin{lem}\label{disjoint_graphing}
$C_\mu(G) = \inf \{ C_\mu(\mathcal{E})\, | \, \mathcal{E} : \text{disjoint graphing of}\, G \}$.
\end{lem}

\begin{proof}
Take an arbitrary graphing $\mathcal{E} = \{ E_i \}_{i \geq 1}$ of $G$. Define a family $\tilde{\mathcal{E}} = \{ \tilde{E}_i \}_{i \geq 1}$ inductively by $\tilde{E}_1 = E_1$ and $\tilde{E}_n = E_n \setminus ( \bigcup_{j=1}^{n-1} \tilde{E}_j)$. Then $\tilde{\mathcal{E}}$ is a disjoint graphing of $G$. We have $C_\mu(\tilde{E}) = \sum_{i \geq 1} \mu^G(\tilde{E}_i) \leq \sum_{i \geq 1} \mu^G(E_i) = C_\mu(\mathcal{E})$, which completes the proof. 
\end{proof}

We need a special simplicial $G$-complex associated with each graphing $\mathcal{E}$ of $G$. In the rest of this subsection, we consider only disjoint graphings. Define $\Sigma_{\mathcal{E}}^{(0)}$ and $\Sigma_{\mathcal{E}}^{(1)}$ as follows:  
\begin{align*}
&\Sigma_{\mathcal{E}}^{(0)} = G; \\
&\Sigma_{\mathcal{E}}^{(1)} = \{ (g_0, g_1) \in \Sigma_\mathcal{E}^{(0)} * \Sigma_\mathcal{E}^{(0)}\, |\, g_0 \neq g_1 \text{and either}\, g_0^{-1} g_1\,  \text{or}\, g_1^{-1} g_0\, \text{belongs to some}\, E \in \mathcal{E} \}.
\end{align*}

\begin{lem}\label{graphing_defines_complex}
The pair $\Sigma_{\mathcal{E}} = ( \Sigma_\mathcal{E}^{(0)}, \Sigma_\mathcal{E}^{(1)} )$ defines a connected, simplicial $G$-complex. Moreover, we have $\alpha_1(\Sigma_\mathcal{E}) = C_\mu(\mathcal{E})$. 
\end{lem}

\begin{proof}
It is easy to see that $\Sigma_\mathcal{E}$ is a simplicial $G$-complex. Since $\mathcal{E}$ is a graphing, the complex $\Sigma_\mathcal{E}$ is connected. Define a subset $F \subset \Sigma_\mathcal{E}^{(1)}$ as follows: $F = \bigcup_{E \in \mathcal{E}} (A_E^+ \cup A_E^-)$ where $A_E^+ := \{ (g_0, g_1) \in \Sigma_\mathcal{E}^{(1)}\, |\, g_0 \in X, \, g_1 \in E \}$ and  $A_E^- := \{ (g_0, g_1) \in \Sigma_\mathcal{E}^{(1)}\, |\, g_1 \in X, \, g_0 \in E \}$. Since each $r \upharpoonright_E$ is injective, we have $A_E^+ \cap A_E^- = \emptyset$ for every $E \in \mathcal{E}$. Thus $F$ is a fundamental domain of $\Sigma_\mathcal{E}^{(1)}$. The disjointness of $\mathcal{E}$ implies that of the family $\{ A_E^+ \cup A_E^- \}_{E \in \mathcal{E}}$. Since $\mu_{\Sigma_{\mathcal{E}}^{(1)}}(A_E^+) = \mu_{\Sigma_{\mathcal{E}}^{(1)}}(A_E^-) = \mu(r(E))$ for every $E \in \mathcal{E}$, we have $\mu_{\Sigma_{\mathcal{E}}^{(1)}}(F^{(1)}) = 2 \sum_{E \in \mathcal{E}} \mu(r(E)) = 2 C_\mu(\mathcal{E})$. On the other hand, since $F^{(1)}$ is a fundamental domain, we have $\mu_{\Sigma_{\mathcal{E}}^{(1)}}(F^{(1)}) = 2 \alpha_1(\Sigma_\mathcal{E})$. Hence we get $\alpha_1(\Sigma_\mathcal{E}) = C_\mu(\mathcal{E})$.
\end{proof}

\begin{rem}
Lemma \ref{graphing_defines_complex} gives another proof of Theorem \ref{treeing_attains_cost}; we can adapt the $\ell^2$-Proof of \cite[Th\'eor\`eme IV. 1]{gaboriau:betti} due to Gaboriau \cite[\S 8]{gaboriau:notes} to arbitrary pmp discrete groupoids. To this end, it suffices to note the last assertion of Lemma \ref{graphing_defines_complex} and the fact that a graphing $\mathcal{E}$ is a treeing if and only if the simplicial complex $(\Sigma_{\mathcal{E}})_x$ is a tree for $\mu$-a.e.~$x \in X$.  
\end{rem}

We are ready to prove Theorem \ref{costvsbetti}.

\begin{proof}(Theorem \ref{costvsbetti})
First, consider the case when $C_\mu(G)$ is finite. Then there exists a graphing $\mathcal{E}$ satisfying $C_\mu(\mathcal{E}) < \infty$. Since $0 \leq \alpha_1(\Sigma_\mathcal{E}) \leq C_\mu(\mathcal{E}) < \infty$, the number $\chi(\Sigma_\mathcal{E})$ is defined and equal to $1 - \alpha_1(\Sigma_\mathcal{E}) = 1 - C_\mu(\mathcal{E})$ by Lemma \ref{graphing_defines_complex}. Thus, by Corollary \ref{euler-poincare}, we have $\beta_0^{(2)}(\Sigma_\mathcal{E}, G) - \beta_1^{(2)}(\Sigma_\mathcal{E}, G) = 1 - C_\mu(\mathcal{E})$. Since $\Sigma_\mathcal{E}$ is connected, Corollary \ref{n-connected} implies that $\beta_0^{(2)}(\Sigma_\mathcal{E}, G) = \beta_0^{(2)}(G)$ and that $\beta_1^{(2)}(\Sigma_\mathcal{E}, G) \geq \beta_1^{(2)}(G)$. Hence we have $\beta_1^{(2)}(G) - \beta_0^{(2)}(G) \leq C_\mu(\mathcal{E}) - 1$ holds for any graphing of finite cost, that is, the desired inequality holds. If $\mathcal{E}$ is a treeing, then we have equality because $\Sigma_{\mathcal{E}}$ is contractible.

Next, consider the case when $C_\mu(G) = \infty$. Then the inequality is trivial since $\beta_0^{(2)}(G)$ is finite as shown below. The number $\beta_0^{(2)}(G)$ equals the $L(G)$-dimension of the module $\Tor_0^{\mathbb{C} [G]} (L(G), L^\infty(X)) \cong L(G) \bigotimes_{\mathbb{C} [G]} L^\infty(X)$, which is a quotient of $L(G)$. Thus, by the additivity of $\dim_{L(G)}$, we conclude that $\beta_0^{(2)}(G) \leq \dim_{L(G)} L(G) = 1$. Hence we are done. If $G$ has a treeing $\mathcal{E} = \{ E_i \}_{i=1}^\infty$, then we have $\beta_1^{(2)}(G) - \beta_0^{(2)}(G) + 1 = \infty$. Indeed, the graphing $\mathcal{E}_i := \{E_1, \dots, E_i \}$ gives a ULB exhaustion $\{ \Sigma_{\mathcal{E}_i} \}_{i \geq 1}$ of $\Sigma_{\mathcal{E}}$. Then, by what we have proved in the previous paragraph, we have $\beta_1^{(2)}(G) - \beta_0^{(2)}(G) + 1 = \lim_{i \to \infty} C_\mu(G_i) = C_\mu(G) = \infty$. Here $G_i$ denotes the groupoid generated by $\mathcal{E}_i$.
\end{proof}

\subsection{An application of the cost versus $L^2$-Betti numbers inequality} 
We have already computed the cost $C_\mu(X \rtimes \mathbb{F}_n)$ directly in Corollary \ref{free_group_action}. We give another explanation of the corollary as an application of Theorem \ref{costvsbetti}. 

To this end, we compute the $L^2$-Betti numbers $\beta_0^{(2)}(X \rtimes \mathbb{F}_n)$ and $\beta_1^{(2)}(X \rtimes \mathbb{F}_n)$. For the case when the action is essentially free, Sauer \cite[Theorem 5.5]{sauer} proved that these $L^2$-Betti numbers coincide exactly with those of the group. Although it is probably well-known that his proof works in the case when the action is not essentially free, we will give its explanation for the sake of completeness. 

Define a $L^\infty (X)$-module $L^\infty(X) \rtimes_{\mathrm{alg}} \Gamma$ as a free $L^\infty(X)$-module with a basis $\{ u_\gamma \}_{\gamma \in \Gamma}$. We endow $L^\infty(X) \rtimes_{\mathrm{alg}} \Gamma$ with a ring structure by requiring $u_\gamma f u_{\gamma^{-1}} = f(\cdot \gamma)$ and $u_\gamma u_{\gamma^\prime} = u_{\gamma \gamma^\prime}$ for $\gamma$, $\gamma^\prime \in \Gamma$ and $f \in L^\infty(X)$. Define a map $\iota : L^\infty(X) \rtimes_{\mathrm{alg}} \Gamma \to \mathbb{C} [ X \rtimes \Gamma ]$ by $\iota( \sum_{\gamma} f_\gamma u_\gamma) := \sum_{\gamma} (f_\gamma \circ r) \mathbbm{1}_{X \times \{ \gamma \}}$. It is not hard to see that $\iota$ is an injective $L^\infty(X)$-module map. 

\begin{lem}\label{inclusion}
The above inclusion $\iota: L^\infty(X) \rtimes_{\mathrm{alg}} \Gamma \to \mathbb{C} [X \rtimes \Gamma]$ is a $\dim_{L^\infty(X)}$-isomorphism.
\end{lem}     

\begin{proof}
It suffices to show the $d_{L^\infty(X)}$-density of $L^\infty(X) \rtimes_{\mathrm{alg}} \Gamma$ in $\mathbb{C} [X \rtimes \Gamma]$. Remark that $\varphi \in \mathbb{C} [X \rtimes \Gamma ]$ belongs to $L^\infty(X) \rtimes_{\mathrm{alg}} \Gamma$ if and only if there exists a finite subset $F \subset \Gamma$ such that $\varphi((x, \gamma)) = 0$ for every  $\gamma \notin F$ and $\mu$-a.e.~$x \in X$. Take $\varphi \in \mathbb{C} [X \rtimes \Gamma ]$. Choose an enumeration $\Gamma = \{ \gamma_i \}_{i \geq 1}$. For every $n \geq 0$, define $X_n := \{ x \in X \, | \, \varphi((x, \gamma_i)) = 0 \, \text{ for every } \, i > n \}$. Then, by the above remark, we have $\mathbbm{1}_{X_n} \varphi \in L^\infty(X) \rtimes_{\mathrm{alg}} \Gamma$. Also we have $\mu(X_n) \to 1$ as $n \to \infty$. Thus, we have $d_{L^\infty(X)} (\mathbbm{1}_{X_n} \varphi, \varphi) = [ \mathbbm{1}_{X_n^c} \varphi ] \leq \mu(X_n^c) \to 0$ as $n \to \infty$. Hence we are done.
\end{proof}

\begin{thm}\label{betti_of_action} (\cite[Theorem 5.5]{sauer})
$\beta_n^{(2)}(\Gamma) = \beta_n^{(2)}( X \rtimes \Gamma)$ holds for every $n \geq 0$.
\end{thm}

\begin{proof}
We have  
\begin{align*}
\beta_n^{(2)}(\Gamma) &= \dim_{L(\Gamma)} \Tor_n^{\mathbb{C} [\Gamma]} (L(\Gamma), \mathbb{C}) \\
&= \dim_{L(X \rtimes \Gamma)} L(X \rtimes \Gamma) \bigotimes_{L(\Gamma)} \Tor_n^{\mathbb{C} [\Gamma]} (L(\Gamma), \mathbb{C}) \qquad (\text{\cite[Theorem 2.6]{sauer}})\\
&= \dim_{L(X \rtimes \Gamma)} \Tor_n^{\mathbb{C} [\Gamma]} (L(X \rtimes \Gamma) \bigotimes_{\mathbb{C} [\Gamma]} L(\Gamma), \mathbb{C}), \qquad (\text{\cite[Theorem 4.3]{sauer}}) 
\end{align*}
equals $\dim_{L(X \rtimes \Gamma)} \Tor_n^{L^\infty(X) \rtimes_{\mathrm{alg}} \Gamma} (L(X \rtimes \Gamma), L^\infty(X))$, since $L^\infty(X) \rtimes_{\mathrm{alg}} \Gamma$ is a free $L^\infty(X)$-module, it is a flat right $\mathbb{C} [\Gamma]$-module. By Lemma \ref{inclusion} and \cite[Lemma 4.8]{sauer}, we can apply \cite[Theorem 4.11]{sauer} to the ring inclusion 
$
L^\infty(X) \subset L^\infty(X) \rtimes_{\mathrm{alg}} \Gamma \subset \mathbb{C} [X \rtimes \Gamma] \subset L(X \rtimes \Gamma)
$.
Then, we get  $\dim_{L(X \rtimes \Gamma)} \Tor_n^{L^\infty(X) \rtimes_{\mathrm{alg}} \Gamma} (L(X \rtimes \Gamma), L^\infty(X)) = \dim_{L(X \rtimes \Gamma)} \Tor_n^{\mathbb{C} [X \rtimes \Gamma ]} (L(X \rtimes \Gamma), L^\infty(X)) = \beta_n^{(2)}( X \rtimes \Gamma)$, which completes the proof.
\end{proof}

We are ready to prove Corollary \ref{free_group_action}. 
\begin{proof}(Corollary \ref{free_group_action})
Applying Theorem \ref{costvsbetti}, we get $\beta_1^{(2)}(X \rtimes \mathbb{F}_n) - \beta_0^{(2)}( X \rtimes \mathbb{F}_n) \leq C_\mu(X \rtimes \mathbb{F}_n) - 1$. Theorem \ref{betti_of_action} and \cite[Example 4.2]{cheeger-gromov} imply that the left hand side is equal to $\beta_1^{(2)}(\mathbb{F}_n) - \beta_0^{(2)}(\mathbb{F}_n) = n - 1$. Therefore, $C_\mu(X \rtimes \mathbb{F}_n) =n$, since $C_\mu(X \rtimes \mathbb{F}_n) \leq n$ is trivial. 
\end{proof}

\begin{rem}
One can give another proof of Lemma \ref{one_element} for the case when $G$ is ergodic in the same way as that in the above proof; we can compute the $L^2$-Betti numbers $\beta_0^{(2)}(G)$ and $\beta_1^{(2)}(G)$ using results of Alekseev-Kyed \cite[Corollary 6.8]{alekseev-kyed} and Sauer-Thom \cite[Corollary 1.4]{sauer-thom}. Probably, it is also possible to prove Lemma \ref{one_element} for the general case in the same way thanks to \cite[Remark 1.7]{sauer-thom}. However, such a proof is more complicated than that we gave in \S\S\S 4.2.3. 
\end{rem}
\subsection*{Acknowledgment}
The author would like to express his gratitude to Professor Yoshimichi Ueda, who is his supervisor for helpful comments and constant encouragement.


\begin{thebibliography}{99}

\bibitem{abert-weiss}M.~Ab\'ert, B.~Weiss, Bernoulli actions are weakly contained in any free action, Ergodic Theory Dynam. Systems, available on CJO2012. doi:10.1017/S0143385711000988.


\bibitem{abert-nikolov}M.~Ab\'ert, N.~Nikolov, Rank gradient, cost of groups and the rank versus Heegaard genus problem, {\it J. Eur. Math. Soc. (JEMS)} 14 (2012), no.5, 1657--1677.


\bibitem{alekseev-kyed}V.~Alekseev, D.~Kyed, Amenability and vanishing of $L^2$-Betti numbers: an operator algebraic approach, {\it J. Funct. Anal.} 263 (2012), no. 4, 1103--1128. 


\bibitem{delaroche-renault} C.~Anantharaman-Delaroche, J.~Renault, {\it Amenable Groupoids}, With a foreword by Georges Skandalis and Appendix B by E. Germain. Monographies de L'Enseignement Math\'ematique, 36. L'Enseignement Math\'ematique, Geneva, 2000.


\bibitem{bermudez}M.~Bermudez, $L^2$-homology for inclusions of von Neumann algebras, arXiv:1403.6044 [math.OA].


\bibitem{cheeger-gromov}J.~Cheeger, M.~Gromov, $L^2$-cohomology and group cohomology, {\it Topology} 25 (1986), no.2, 189--215.


\bibitem{connes-shlyakhtenko}A.~Connes, D.~Shlyakhtenko, $L^2$-homology for von Neumann algebras, {\it J. Reine Angew. Math.} 586 (2005), 125--168. 


\bibitem{feldman-moore}J.~Feldman, C.~Moore, Ergodic equivalence relations, cohomology, and von Neumann algebras. II, {\it Trans. Amer. Math. Soc.} 234 (1977), no. 2, 325--359.


\bibitem{gaboriau:cost}D.~Gaboriau, Co\^ut\ des relations d'\'equivalence et des groupes, {\it Invent. Math.} 139 (2000), no.1, 41--98.


\bibitem{gaboriau:betti}D.~Gaboriau, Invariants $l^2$ de relations d'\'equivalence et de groupes, {\it Publ. Math.Inst. Hautes \'Etudes Sci.} No.95 (2002), 93--150.


\bibitem{gaboriau:notes}D.~Gaboriau, Around the orbit equivalence theory of the free groups, Lecture notes for the Masterclass on Ergodic Theory and von Neumann algebras at the University of Copenhagen, 2013, available at the webpage of the masterclass: \url{http://www.math.ku.dk/english/research/conferences/2013/masterclass_ergodic/}    


\bibitem{kadison-ringrose:bookv2}R.~ V.~Kadison, J.~R.~Ringrose, {\it Fundamentals of the theory of operator algebras. Vol. II. Advanced theory}, Corrected reprint of the 1986 original. Graduate Studies in Mathematics, 16. American Mathematical Society, Providence, RI, 1997.


\bibitem{kechris:GTM156}A.~S.~Kechris, {\it Classical Descriptive Set Theory}, G.T.M, 156. Springer-Verlag, New York, 1995. 


\bibitem{kosaki}H.~Kosaki, Free products of measured equivalence relations, {\it J. Funct. Anal.} 207 (2004), no.2, 264--299.


\bibitem{levitt}G.~Levitt, On the cost of generating an equivalence relation, {\it Ergodic Theory Dynam. Systems}, 15 (1995), no.6, 1173--1181. 


\bibitem{luck_98}W.~L\"uck, Dimension theory of arbitrary modules over finite von Neumann algebras and $L^2$-Betti numbers. I. Foundations. (English summary) {\it J. Reine Angew. Math.} 495 (1998), 135--162. 


\bibitem{luck:survey}W.~L\"uck, {\it $L^2$-invariants: theory and applications to geometry and K-theory}, A Series of Modern Surveys in Mathematics, 44. Springer-Verlag, Berlin, 2002.


\bibitem{neshveyev}S.~Neshveyev, S.~Rustad, On the definition of $L^2$-Betti numbers of equivalence relations, {\it Internat. J. Algebra Comput.} 19 (2009), no.3, 383--396.


\bibitem{pichot}M.~Pichot, Harmonic analysis from quasi-periodic domains, {\it Israel J. Math.}167, (2008), 63--90.


\bibitem{sauer}R.~Sauer, $L^2$-Betti numbers of discrete measured groupoids, {\it Internat. J. Algebra Comput.} 15 (2005), no. 5-6, 1169--1188.


\bibitem{sauer-thom} R.~Sauer, A.~Thom, A spectral sequence to compute $L^2$-Betti numbers of groups and groupoids, {\it J. Lond. Math. Soc.} (2) 81 (2010), no. 3, 747--773.


\bibitem{spanier}E.~H.~Spanier, {\it Algebraic Topology}, Corrected reprint. Springer-Verlag, New York-Berlin, 1981.


\bibitem{thom}A.~Thom, $L^2$-invariants and rank metric, in {\it $C^*$-algebras and Elliptic Theory II, Trends Math.}, 267--280, Trends Math., Birkh\"auser, Basel, 2008.


\bibitem{ueda}Y.~Ueda, Notes on treeability and costs for discrete groupoids in operator algebra framework, {\it Operator Algebras: The Abel Symposium 2004}, 259--279.


\bibitem{voiculescu-dykema-nica}D.~V.~Voiculescu, K.~J.~Dykema, A.~Nica, {\it Free Random Variables}, CRM Monograph Series, 1. American Mathematical Society, Providence, RI, 1992.  
\end{thebibliography}
\end{document}